\documentclass[10pt]{siamltex}
\usepackage{graphics,color}
\usepackage{amsmath,amsfonts,amscd,amssymb,bm}
\usepackage{mathrsfs}
\usepackage{fancyhdr}
\usepackage{listings}
\usepackage{url}
\usepackage{enumerate}
\usepackage[pdftex]{graphicx}
\usepackage{paralist}
\usepackage[section]{placeins}
\usepackage{mathptmx}
\usepackage{marginnote}
\allowdisplaybreaks
\numberwithin{equation}{section}

\def\R{{\mathbb R}}

\def\domf{\Omega}
\def\g{\Gamma}
\def\velf{\boldsymbol{u}}

\def\tf{\boldsymbol \varphi}

\def\ts{\boldsymbol{\xi}}

\def\tp{q}

\def\str{\boldsymbol{\sigma}}
\def\du{\boldsymbol{D}}

\def\norm{\boldsymbol{n}}

\def\displ{\boldsymbol\eta}

\def\disv{\boldsymbol{v}}

\def\disvt{\tilde{\boldsymbol{v}}}

\def\f{\boldsymbol{f}}
\def\dt{\partial_t}
\def\dtt{\partial_{tt}}
\def\n{\|}
\def\dx{d \boldsymbol{x}}

\def\velfh{\boldsymbol{u}_h}
\def\velfth{\tilde{\boldsymbol{u}}_h}

\def\tfh{\boldsymbol{\varphi}_h}
\def\tvh{\boldsymbol{\psi}_h}

\def\tsh{\boldsymbol{\xi}_h}

\def\tph{q_h}

\def\displh{\boldsymbol\eta_h}

\def\disvh{\boldsymbol{v}_h}

\def\disvth{\tilde{\boldsymbol{v}}_h}

\def\thf{\boldsymbol\theta_f}
\def\defl{\boldsymbol\delta_f}
\def\ths{\boldsymbol{\theta}_s}
\def\des{\boldsymbol\delta_s}
\def\thp{\theta_p}
\def\dep{\delta_p}
\def\thv{\boldsymbol\theta_v}
\def\dev{\boldsymbol\delta_v}
\def\thvt{\boldsymbol\theta_{\tilde{v}}}
\def\devt{\boldsymbol\delta_{\tilde{v}}}

\def\ltt{\lesssim}
\def\gtt{\gtrsim}

\newtheorem{remark}{Remark}

\begin{document}

\title{Stability and convergence analysis of the extensions of the kinematically coupled scheme for the fluid-structure interaction}

\author{Martina Bukac\thanks{%
Department of Applied and Computational Mathematics and Statistics, University of Notre Dame, Notre Dame, IN 46556, USA. email: \texttt{mbukac@nd.edu}. Partially supported by the NSF under grants DMS 1318763 and DMS 1619993. (Corresponding author.)} \and Boris Muha \thanks{%
Department of Mathematics, University of Zagreb, 10000 Zagreb,
Croatia. email: \texttt{borism@math.hr}. Partially supported by Croatian Science Foundation grant
number 9477 and by the NSF under grant NSF DMS 1311709.}}

\maketitle


\begin{abstract}
In this work we analyze the stability and convergence properties of a loosely-coupled scheme, called the kinematically coupled scheme, and its extensions for the interaction between an incompressible, viscous fluid and a thin, elastic structure. We consider a benchmark problem where the structure is modeled using a general thin  structure model, and the coupling between the fluid and structure is linear.
We derive the energy estimates associated with the unconditional stability of an extension of the kinematically coupled scheme, called the $\beta$-scheme. Furthermore, for the first time  we present \textit{a priori} estimates showing  optimal, first-order in time convergence in the case when $\beta=1$. We further discuss the extensions of our results to other fluid-structure interaction problems,  in particular the fluid-thick structure interaction problem. The theoretical stability and convergence results are supported with numerical examples.

\end{abstract}



\section{Introduction}

The interaction between an incompressible viscous fluid and  an elastic structure has been of great interest due to  various applications in different areas (see e.g. \cite{bodnar2014fluid}). This problem is characterized by highly non-linear coupling between two different physical phenomena.  As a result, a comprehensive study of such problems remains a challenge~\cite{hou2012numerical}. The solution strategies for fluid-structure interaction (FSI) problems can be roughly classified as monolithic schemes and loosely or strongly coupled partitioned schemes. Monolithic algorithms, see for example \cite{bazilevs2008isogeometric,gerbeau2003quasi,
nobile2001numerical,gee2011truly,ryzhakov2010monolithic,
hron2006monolithic}, consist of solving the entire coupled problem as one system of algebraic equations.
They, however, require well-designed preconditioners~\cite{gee2011truly,badia2008modular,heil2008solvers} and are generally quite expensive in terms of computational time and memory requirements. 
Hence, to obtain smaller and better conditioned sub-problems, reduce the computational cost and treat each physical phenomenon separately, partitioned numerical schemes that solve the fluid problem separately  from the structure problem have been a popular choice.  The development of partitioned numerical methods for FSI problems has been extensively studied~\cite{degroote2008stability,degroote2011similarity,farhat2006provably,
bukavc2012fluid,badia2009robin,SunBorMulti,
Fernandez3,nobile2008effective,hansbo2005nitsche,
lukavcova2013kinematic,
banks2014analysis,banks2014analysis2,fernandez2015convergence}, but the design of efficient schemes to produce stable, accurate results remains a challenge. Moreover, despite the recent developments, there are just a few works where the convergence is proved rigorously \cite{BorSun,SunBorMulti,Fernandez3,fernandez2015convergence}.

A classical partitioned scheme, particularly popular in aerodynamics, is known as the  Dirichlet-Neumann (DN) partitioned scheme~\cite{causin2005added,nobile2008effective,forster2007artificial}. The  DN scheme consists of solving the fluid problem with a Dirichlet boundary condition (structure velocity)  at the fluid-structure interface, and the structure problem with a Neumann boundary condition  (fluid stress) at the interface. While the  DN scheme features appealing properties such as modularity, simple implementation and fast computational time, it  has been shown to be stable only if the structure density is much larger than the fluid density. This requirement is easily  achieved in some applications like aerodynamics, but not in other applications like hemodynamics where the density of blood is of the same order of magnitude as the density of  arterial walls.  
In these cases,  the energy of the discrete problem in the  DN partitioned algorithm does not accurately approximate  the energy of the continuous problem, introducing numerical instabilities known as \textit{the added mass effect}~\cite{causin2005added}. A partial solution to this problem is to sub-iterate the fluid and structure sub-problems at each time step until the energy at the fluid-structure interface is balanced. However, schemes that require sub-iterations, also known as strongly coupled schemes, are computationally expensive and may suffer from convergence issues for certain parameter values~\cite{causin2005added,forster2007artificial}.

To circumvent these difficulties, and to retain the main advantages of partitioned  schemes,  several new  algorithms have been proposed. Methods proposed in~\cite{colciago2014comparisons,figueroa2006coupled,nobile2008effective} use a membrane model for the structure that is then embedded into the fluid problem where it appears as a generalized Robin boundary condition. In addition to the classical Dirichlet-Neumann and Neumann-Dirichlet schemes, Robin-Neumann and a Robin-Robin algorithms, that converge without relaxation and need a smaller number of sub-iterations between the 
fluid and the structure, provided that the interface parameters are suitably chosen, have been proposed in~\cite{badia2009robin,badia2008fluid,gerardo2010analysis}. 
Karniadakis et al.~\cite{baek2012convergence,yu2013generalized} proposed fictitious-pressure and fictitious-mass algorithms, in which  the added mass effect is accounted for by incorporating additional terms into governing equations.
However, algorithms proposed in~\cite{badia2009robin,badia2008fluid,nobile2008effective,baek2012convergence,yu2013generalized}  require sub-iterations between the fluid and the structure sub-problems in order to achieve stability.   
A different approach based on Nitsche's penalty method~\cite{hansbo2005nitsche} was used in~\cite{burman2009stabilization,burman2013unfitted}. The formulation in~\cite{burman2009stabilization,burman2013unfitted}
still suffers from stability issues, which were corrected by adding a weakly consistent stabilization term
that includes pressure variations at the interface. The splitting error, however, lowers the temporal
accuracy of the scheme, which was then corrected by proposing a few defect-correction sub-iterations
to achieve an optimal convergence rate. Recently,  so called \emph{added-mass partitioned schemes} were proposed in~\cite{banks2014analysis,banks2014analysis2}. Using the von Neumann stability analysis, the authors showed that the algorithm proposed in~\cite{banks2014analysis} is weakly stable under a Courant--Friedrichs--Lewy (CFL) condition, while the algorithm proposed in~\cite{banks2014analysis2} is stable under a condition on the time step which depends on the structure parameters. Even though the authors do not derive the convergence rates, their numerical results indicate that both schemes are second-order accurate in time.

A loosely-coupled numerical scheme, called the ``kinematically coupled scheme'', was introduced
in~\cite{guidoboni2009stable}. The scheme is based on the Lie operator splitting, where the 
fluid and the structure sub-problems are fully decoupled and communicate only via  the interface conditions. More precisely, in each time-step the initial interface velocity in the structure sub-problem is taken from the fluid sub-problem and vice versa.
Due to the appealing features of the kinematically coupled scheme, such as modularity, stability, and easy implementation, several extensions have been proposed that include modeling FSI between artery, blood flow, and a cardiovascular device called a stent~\cite{BorSunStent}, FSI with thick structures~\cite{thick}, FSI with composite structures~\cite{multilayered}, FSI with poroelastic structures~\cite{bukavc2012fluid}, and FSI involving non-Newtonian fluids~\cite{Lukacova,lukavcova2013kinematic}. The kinematically coupled scheme has been shown to be unconditionally stable, circumventing instabilities associated with the added mass effect~\cite{guidoboni2009stable,SunBorMar,Fernandez3}. However, its  order of temporal convergence  is only $\mathcal{O}(\sqrt{\Delta t})$ \cite{BorSun,Fernandez3}, and hence sub-optimal. In order to improve the accuracy, the extension of the kinematically coupled scheme, so-called $\beta$-scheme, was introduced by the authors in~\cite{bukavc2012fluid} and ``the incremental displacement correction scheme'' was  proposed by Fernandez in~\cite{Fernandez3}.  Better accuracy was achieved in~\cite{bukavc2012fluid} by introducing a parameter $\beta$ which controls the amount of the fluid pressure used to load the structure sub-problem. In~\cite{Fernandez3} the accuracy is improved by treating the structure explicitly in the fluid sub-problem and then correcting it in the structure sub-problem. A more detailed comparison between these two basic extensions  of the kinematically coupled scheme is given in Section \ref{comparison}. While the incremental displacement  correction scheme is supported by the stability and convergence analysis, the improved accuracy of $\beta$-scheme had only been observed numerically~\cite{bukavc2012fluid}.

The goal of this work is to understand the mechanism which leads to a better accuracy and prove the optimal convergence result for the $\beta$-scheme. We show that the optimal convergence rate is achieved when $\beta=1$,  in which case the structure is loaded with the full fluid stress. The main result of the paper is Theorem \ref{MainThm}, in which we derive the error estimates of the fully discrete problem. Our estimates prove the optimal, first-order convergence in time and optimal convergence in space. The results are obtained assuming that the structure undergoes infinitesimal displacements. In this case, the coupling between the fluid and structure is linear. This is a standard assumption in the convergence and stability analysis of the FSI problems (see e.g. \cite{causin2005added,Fernandez3}) because  the ``added-mass'' effect and stability issues connected to it are already present in the linear case. Even though the analysis in the paper is performed on a linear problem, the main results are numerically tested and confirmed  on the full non-linear problem.

This paper is organized as follows: We introduce the linear fluid-structure interaction model in Section~2, deriving the weak formulation of the monolithic problem. The numerical scheme is presented  in Section~3, while the comparison with the alternative scheme proposed in \cite{Fernandez3} is given in Section \ref{comparison}. The energy estimates associated with the unconditional stability are derived in Section~4. In Section~5 we derive the \textit{a priori} energy estimates and prove  first-order convergence in time. In Section~\ref{Sec:thick} we generalize the obtained result to the cases  when the structure is thick or multi-layered. Theoretical results from Sections~5 and 6 are supported by the numerical experiments in Section~7. Finally, conclusions are drawn in Section~8.

\section{Description of the problem}

We consider a  linear fluid-structure interaction problem where the structure is described by some lower dimensional, linearly elastic model (for example membrane, shell, plate, etc). 
In the cases of  nonlinear, moving boundary FSI problems, even the question of existence of a solution is challenging and we refer the reader to \cite{BorSun} and references within.

Let $\Omega\subset\R^d$, $d=2,3$, be an open, smooth set and $\partial\Omega=\overline{\Sigma}\cup\overline{\Gamma}$, where $\Gamma$ represents elastic part of the boundary while $\Sigma$ represents artificial (inflow or outflow) of the boundary (see Figure~\ref{domain}). We assume  that the structure undergoes infinitesimal displacements, and that the fluid is incompressible, Newtonian, and is characterized by a laminar flow regime.
 \begin{figure}[ht!]
 \centering{
 \includegraphics[scale=1.3]{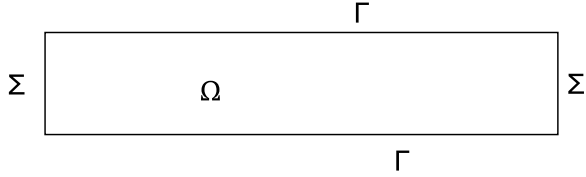}
 }
 \caption{Fluid domain $\Omega$. The lateral boundary $\Gamma$ represents elastic structure. }
\label{domain}
 \end{figure}

Thus, we model the fluid by the time-dependent Stokes equations in a fixed domain $\domf$
\begin{align}\label{Stokes1}
& \rho_f \dt \velf = \nabla \cdot \str (\velf,p), \qquad \nabla \cdot \velf = 0 & \textrm{in}\; \domf \times (0,T), \\
&\str (\velf,p) \norm=-p_{in/out}(t)\norm & \textrm{on}\; \Sigma\times (0,T), \label{Neumann}\\
&\velf(.,0)={\bf u}^0  &  \textrm{in}\;\domf,
\end{align}
where $\velf=(u_i)_{i=1,\dots,d}$ is the fluid velocity, $\str  (\velf,p)= -p \boldsymbol{I} + 2 \mu \du(\velf)$  is the fluid stress tensor, $p$ is the fluid pressure, $\rho_f$ is the fluid density, 
$\mu$ is the fluid viscosity,  $\norm$ is the outward normal to the fluid boundary, $p_{in/out}$ is the prescribed inflow or outflow pressure and  $\du(\velf) = (\nabla \velf+(\nabla \velf)^{T})/2$ is the strain rate tensor. 

\begin{remark}
We could also prescribe other types of boundary conditions on  various parts of $\Sigma$, for example symmetry boundary condition, slip boundary condition or no-slip boundary condition. These types of mixed boundary conditions do not effect our analysis. However since we are interested in simulating a pressure-driven flow and in order to keep the notation simple, we choose to work only with boundary condition~\eqref{Neumann}.
\end{remark}

The lateral boundary represents a  thin, elastic wall whose dynamics is modeled by some linearly elastic lower-dimensional model, given by
\begin{align}
 & \rho_s \epsilon \dtt \displ+{\mathcal L}_s\displ= {\bf f}  &  \textrm{on} \; \Gamma\times (0,T), \label{mem} \\
 & \displ(.,0)=\displ^0,\; \dt\displ(.,0)={\bf v}^0 & \textrm{on} \; \Gamma, \label{strucinit}
\end{align}
where $\displ= (\eta_i)_{i=1,\dots,d}$ denotes the structure displacement, $\f$ is a vector of surface density of the force applied to the thin structure, $\rho_{s}$ denotes the structure density and $\epsilon$ denotes the structure thickness. Moreover, we define a  bilinear form associated with the structure operator
$$a_s(\displ,\ts)=\int_{\Gamma}{\mathcal L}_s\displ\cdot\ts dS \quad \textrm{and norm} \;\; \|\displ\|^2_S=a_s(\displ,\displ).$$
  We assume that operator ${\mathcal L}_s$ is such that norm $\|.\|_S$ is equivalent to the $H^1(\Gamma)$ norm. One example of such operator is the one associated with the linearly elastic cylindrical Koiter shell used in \cite{bukavc2012fluid}.
Finally, we prescribe clamped boundary conditions for the thin structure:
\begin{equation}\label{clamped}
 \displ(0,t)=\displ(L,t) = 0, \quad \textrm{for} \; t\in(0,T).
\end{equation}
The fluid and the structure are coupled via the kinematic and dynamic boundary conditions: 
\begin{itemize}
\item[ ] \textit{The kinematic coupling condition (continuity of velocity):} $ \velf = \dt \displ \quad \textrm{on} \; \g \times(0,T) $.
\item[ ] \textit{The dynamic coupling condition (balance of contact forces):} $
\f=-\str (\velf,p) \norm \quad \textrm{on} \; \g \times (0,T).$
\end{itemize}


\subsection{Weak formulation of the monolithic problem}
 For a domain $A$, we denote by $H^k(A)$ the standard Sobolev space and $L^2(A)$ the standard space of square integrable functions. These are Hilbert spaces and we denote by $\n \cdot \n_{H^k(A)}$ and $\n\cdot \n_{L^2(A)}$ the corresponding norms.
\begin{gather}
V^f =(H^1(\domf))^d, \quad  Q^f =  L^2(\domf),  \quad V^s = (H_0^1(\g))^d, \quad
V^{fsi} =  \{ (\tf, \ts) \in V^f \times V^s | \ \tf|_{\g}  = \ts \}, \notag
\end{gather}
for all $t \in [0,T)$, and introduce the following bilinear forms
\begin{gather}\label{FluidBilinear}
 a_f(\velf, \tf) = 2 \mu \int_{\domf} \du(\velf) : \du(\tf) d \boldsymbol x,  \qquad  b(p, \tf) =  \int_{\domf} p \nabla \cdot \tf d \boldsymbol x.
 \end{gather}
We define  norm $\n \cdot \n_F$ associated with the fluid  bilinear form as $\n \velf \n_F : = \n \du(\velf) \n_{L^2(\domf)}, \;\; \forall \velf \in V^f.$
 
The variational formulation of the monolithic fluid-structure interaction problem now reads: given $t \in (0, T)$ find $(\velf, \displ, p) \in  V^f  \times V^s \times Q^f$ with $\velf = \partial_t \displ$ on $\g$, such that for all $(\tf, \ts, \tp) \in V^{fsi} \times Q^f$
\begin{gather}
\rho_f \int_{\domf} \dt \velf \cdot \tf \dx +a_f(\velf,\tf)-b(p, \tf)+b(\tp, \velf) +  \rho_{s} \epsilon\int_{\g} \dtt \displ\cdot \ts dx  
+a_s(\displ, \ts) =  \int_{\Sigma} p_{in/out}(t) \tf\cdot\norm dS. \label{weakmonolithic}
\end{gather}

\section{The numerical scheme}\label{scheme}

To solve the fluid-structure interaction problem presented in Section 2, we use a loosely coupled numerical scheme, called the kinematically coupled $\beta$ scheme.  The scheme is based on an operator splitting method called  Lie splitting~\cite{glowinski2003finite}, which separates the original problem into a fluid sub-problem and a structure sub-problem. The equations are split in a way such that the fluid problem is solved with a Robin-type boundary condition including the structure inertia.  As we shall show later, this is the main key in proving the stability of the scheme. The structure sub-problem is loaded by a part of the fluid normal stress obtained from the previous time step.  The amount of stress  applied to the structure is measured by a parameter $\beta\in [0,1]$. Namely, we split the normal fluid stress as
$$ \str \norm  =  \underbrace{\str \norm - \beta \str \norm}_{\textrm {Part I}}  \underbrace{+\beta \str \norm}_{\textrm{Part II}}. $$
Part II in the equation above is used to load the thin structure, while Part I gives rise to a Robin-type boundary condition for the fluid sub-problem.  

The case $\beta=0$ corresponds to the classical kinematically coupled scheme which was introduced in \cite{guidoboni2009stable}, where  in each time-step the fluid and structure sub-problems communicate only via the initial guesses for the interface conditions.
Namely, the structure elastodynamics is driven only by the initial velocity, setting it equal to the fluid velocity from the previous time step. Including some loading from the fluid, as done in~\cite{Martina_jcp}, was shown to increase the accuracy of the scheme.  The loading on the structure used in~\cite{Martina_jcp} was introduced in a similar fashion as here, but instead of loading the structure with the fluid normal stress, it was loaded only by the fluid pressure. This was done because the algorithm presented there was motivated by  biomedical applications (blood flow through the compliant vessels), where the pressure is the leading order term of the fluid stress. However, as we will see later, for theoretical reasons here we take into  account the full normal stress.


Let $t^n:=n \Delta t$ for $n = 1, \ldots, N,$ where $T=N \Delta t$ is the final time. To discretize the problem in time, we use the Backward Euler scheme. We denote the discrete time derivative by $d_t \tf^{n+1} = \Delta t^{-1}(\tf^{n+1}-\tf^{n})$.

The kinematically coupled $\beta$ scheme for the time-discrete problem is given as follows (see~\cite{guidoboni2009stable,Martina_jcp} for details):
\begin{itemize}
\item \textbf{Step 1: The structure sub-problem.} 
Find $ \disvt^{n+1},$ and $\displ^{n+1}$ such that  
\begin{align}
& \rho_s \epsilon \frac{\disvt^{n+1}-\disv^{n}}{\Delta t}+{\mathcal L}_S\displ^{n+1}= -\beta \str (\velf^n,p^n) \norm   &  \textrm{on} \; \Gamma,  \label{Step1F1} \\
 &d_t \displ^{n+1} = \disvt^{n+1} & \textrm{on} \; \Gamma,\label{Step1F2}
\end{align} 
with boundary conditions: 
\begin{equation}
 \displ^{n+1}(0)=\displ^{n+1}(L) = 0. \label{lcend}
\end{equation}
The structure velocity computed in this sub-problem is then used as an initial condition in Step 2. Note that the velocity of the fluid does not change in this step. 
\\ 
 \item \textbf{Step 2. The fluid sub-problem.}
Find $\velf^{n+1}, p^{n+1}$ and $\disv^{n+1}$ such that 
\begin{align}
&\rho_f d_t \velf^{n+1} = \nabla \cdot \str (\velf^{n+1},p^{n+1}) & \textrm{in}\; \domf, \label{lcbeg}\\
& \nabla \cdot \velf^{n+1} = 0 & \textrm{in} \; \domf, \label{DivC} \\
& \rho_s \epsilon \frac{\disv^{n+1}-\disvt^{n+1}}{\Delta t}= -\str(\velf^{n+1},p^{n+1})  \norm  + \beta \str(\velf^{n},p^{n})  \norm  &  \textrm{on} \; \Gamma, \label{robin1}  \\
   &  \velf^{n+1}=\disv^{n+1}   & \textrm{on} \; \Gamma, \label {kin}
\end{align}
with the following boundary conditions on $\Sigma$: 
\begin{gather}
 \str (\velf^{n+1},p^{n+1})  \norm = -p_{in/out}(t^{n+1}) \norm \; {\rm on}\ \Sigma,
\end{gather}
and the initial conditions obtained in Step 1. 

Do $t^n=t^{n+1}$ and return to Step 1.
\end{itemize}

\begin{remark}
Combining equation~\eqref{kin} with equation~\eqref{robin1}  gives rise to a Robin-type boundary condition for the fluid velocity.
The structure displacement remains unchanged in this step. 
\end{remark}

To discretize the problem in space, we use the finite element method based on a conforming FEM triangulation with maximum triangle diameter $h$. Thus, we introduce the finite element spaces $V^f_h \subset V^f, Q^f_h \subset Q^f$, and $V^s_h \subset V^s$.  
The fully discrete numerical scheme in the weak formulation is given as follows:
\begin{itemize}
\item \textbf{Step 1.} Given $t^{n+1} \in (0,T], n = 0, \ldots, N-1,$  find $\disvth^{n+1} \in V_h^s $, with  $d_t \displh^{n+1}=\disvth^{n+1}$,  such that for all $ \tsh \in V^s_h $ we have 
 \begin{gather}
\rho_{s} \epsilon \int_{\Gamma} \frac{\disvth^{n+1}-\disvh^{n}}{\Delta t} \cdot \tsh dS + a_s (\displh^{n+1}, \tsh)
 =-\beta\int_{\Gamma} \str (\velfh^{n},p_h^{n})  \norm \cdot \tsh dS.\label{S1discrete}
\end{gather}
\item \textbf{Step 2.} Given $\disvth^{n+1}$ computed in Step 1, find $(\velfh^{n+1}, \disvh^{n+1}) \in V_h^f\times V^s_h, $ with $\velfh^{n+1}|_{\g} = \disvh^{n+1},$ and $p_{h}^{n+1} \in Q_h^f$ such that for all $(\tfh,\tvh,\tph) \in V_h^f \times V_h^s \times Q_h^f$, with $\tfh|_{\g} = \tvh$, we have
\begin{gather}
\rho_f \int_{\domf} d_t \velfh^{n+1} \cdot \tfh \dx+  a_f(\velfh^{n+1}, \tfh)- b(p^{n+1}_{h}, \tfh) +b(\tph, \velfh^{n+1})+\rho_{s} \epsilon \int_{\Gamma} \frac{\disvh^{n+1}-\disvth^{n+1}}{\Delta t} \cdot \tvh dS  \nonumber \\
 =
 \beta\int_{\Gamma} \str (\velfh^{n},p_h^{n})  \norm \cdot \tvh dS 
+\int_{\Sigma} p_{in/out}(t^{n+1}) \tfh\cdot\norm dS. \label{S2discrete}
 \end{gather}
\end{itemize}

\subsection{Comparison of the kinematically coupled $\beta$ scheme and the incremental displacement-correction scheme}\label{comparison}

In this section we
illustrate the differences between the kinematically coupled $\beta$ scheme \cite{Martina_jcp} and the incremental displacement-correction scheme~\cite{Fernandez3}.
It was proven in \cite{BorSun} that the  original kinematically coupled scheme (case $\beta=0$) applied to the full, nonlinear moving boundary FSI problem is convergent. Moreover, even though not explicitly stated, it was proven that the splitting error is of order at most $\sqrt{\Delta t}$ (\cite{BorSun}, formula (67) and proof of Theorem 2). The same was proven in~\cite{Fernandez3} for a linear problem (see \cite{Fernandez3}, Theorem 5.2).

We first consider the $\beta$ scheme and sum  equations~\eqref{Step1F1} and \eqref{robin1}, and use \eqref{Step1F2}, \eqref{lcbeg}, \eqref{DivC}, \eqref{kin}.
To shorten the notation in this section, we denote $\str^{n}:=\str(\velf^{n},p^{n}), \; \forall n.$
Variables   $\velf^{n+1}$, $\disv^{n+1}$ and $\displ^{n+1}$ satisfy the following equations:
\begin{align}
&\rho_f d_t \velf^{n+1} = \nabla \cdot \str^{n+1} & \textrm{in}\; \domf,\\
& \nabla \cdot \velf^{n+1} = 0 & \textrm{in} \; \domf, \\
& \rho_s \epsilon \frac{\disv^{n+1}-\disv^{n}}{\Delta t}+{\mathcal L}_S\displ^{n+1}= -(\str^{n+1} \norm)      &  \textrm{on} \; \Gamma,   \\
   &  \disv^{n+1}=\velf^{n+1}  & \textrm{on} \; \Gamma, \\
  &d_t \displ^{n+1} = \disv^{n+1}+(\disvt^{n+1}-\disv^{n+1}) & \textrm{on} \; \Gamma.\label{SplttingErr}
\end{align}

Notice that this is exactly the monolithic formulation of the considered FSI problem \eqref{Stokes1}-\eqref{clamped} with an additional term in~\eqref{SplttingErr}. Therefore, term $(\disvt^{n+1}-\disv^{n+1})$ accounts for the splitting error. 
From~\eqref{robin1} we obtain
\begin{equation}\label{SplittingErr2}
\disvt^{n+1}-\disv^{n+1}=\frac{\Delta t}{\varrho_s\epsilon}\big (\str^{n+1} \norm - \beta \str^{n} \norm\big )=\frac{\Delta t}{\varrho_s\epsilon} \Big [\beta\big ( \str^{n+1} \norm-\str^{n} \norm \big )+(1-\beta)\str^{n+1} \norm \Big ] \quad \textrm{on} \; \g.
\end{equation}
The right hand side of \eqref{SplittingErr2} consists of two terms, one involving $\str^{n+1} \norm -  \str^{n}\norm$ and the other involving $\str^{n+1} \norm$. From the Taylor expansion, one can see that the first term will have first order accuracy in time, while no such estimate can be obtained for the second term. Therefore,  the choice $\beta=1$  yields the smallest splitting error because the last term will equal zero. Hence, the main goal in our analysis is to take advantage of the correction made by the fluid stress (with $\beta=1$) in order to get better estimates of the splitting error term which yield optimal convergence rate.

In order to remedy the problem of sub-optimal accuracy, Fernadez~\cite{Fernandez3} proposed a different extension of the kinematically coupled scheme, so-called ``incremental displacement-correction" scheme. In  the first step of this scheme, one solves the FSI problem with the explicit treatment of the structure elasticity operator  ${\mathcal L}_S\displ^{n}$, correcting it  in the second step. Instead of adding and subtracting the normal stress from the previous time step, which leads to the $\beta$ scheme, the incremental displacement-correction scheme is obtained by adding and subtracting the elastic operator ${\mathcal L}_S\displ^{n}$ applied to the displacement from the previous time step. This scheme can also be viewed as a kinematic perturbation of the monolithic scheme in the following way. Let $\velf^{n+1}$, $\disv^{n+1}$ and $\displ^{n+1}$ be the fluid velocity, the structure velocity, and the structure displacement, respectively, obtained in $n+1$th step of the incremental displacement-correction scheme. Then, they satisfy the following equations:
\begin{align}
&\rho_f d_t \velf^{n+1} = \nabla \cdot \str^{n+1} & \textrm{in}\; \domf,\\
& \nabla \cdot \velf^{n+1} = 0 & \textrm{in} \; \domf, \\
& \rho_s \epsilon \frac{\disv^{n+1}-\disv^{n}}{\Delta t}+{\mathcal L}_S\displ^{n+1}= -(\str^{n+1} \norm)     &  \textrm{on} \; \Gamma,   \\
   &  \disv^{n+1}+(\disvt^{n+1}-\disv^{n+1})=\velf^{n+1} & \textrm{on} \; \Gamma, \\
  &d_t \displ^{n+1} = \disv^{n+1} & \textrm{on} \; \Gamma,\label{SplttingFer}
\end{align}
Again, we see that term $(\disvt^{n+1}-\disv^{n+1})$ accounts for the splitting error, but in this case the splitting error is manifested as the error in the kinematic coupling condition. Fernandez showed that this scheme has an optimal, first-order convergence in time (\cite{Fernandez3}, Theorem 5.2).

To summarize, there are two different extensions of the kinematically coupled scheme presented in the literature, both introduced to improve the accuracy. Both of them correct the splitting error, but in a different manner. The $\beta$ scheme first solves the structure problem with the forcing from the fluid computed in the previous time step. Then, it solves the fluid problem with a Robin-type boundary condition involving the structure inertia. 
On the other hand, in the incremental displacement-correction scheme one first solves the whole FSI problem with the explicit treatment of the elastic operator, and then in the second step corrects the structure displacement. Both scheme have the structure inertia included in the fluid step which is crucial for the stability. 

\section{Stability analysis}\label{sec:Stability}

In this section we derive an energy estimate that is associated with unconditional stability of algorithm~\eqref{S1discrete}-\eqref{S2discrete}. Based on our previous results~\cite{Martina_jcp} and arguments in Section~\ref{comparison}, we expect the optimal accuracy when $\beta=1$. Namely, when $0\leq \beta<1$ we have additional term in the splitting error \eqref{SplttingErr} which causes suboptimal convergence rate of order $1/2$. However, as we show in the Section with numerical experiments, in practical computations this term can be small for $\beta$ close to 1. Hence, from here on we use $\beta=1$ in our analysis. 

Let $a \ltt(\gtt) b$ denote that there exists a positive constant $C$, independent of the mesh size $h$ and the time step size $\Delta t$, such that $a \le (\ge) C b$.
Let $\mathcal{E}_f(\velfh^n)$ denote the discrete kinetic energy of the fluid, $\mathcal{E}_v(\disvh^n)$ denote the discrete kinetic energy of the structure, and $\mathcal{E}_s(\displh^n)$ denote the discrete elastic energy of the structure at time level $n$, defined respectively by
\begin{gather}
 \mathcal{E}_f(\velfh^n) = \frac{\rho_f}{2} \n\velfh^n\n^2_{L^2(\domf)}, \quad \mathcal{E}_v (\disvh^n)= \displaystyle\frac{\rho_{s} \epsilon}{2} \n \disvh^{n}\n^2_{L^2(\Gamma)}, \quad  \mathcal{E}_s (\displh^n) =\frac{1}{2} \n \displh^n \n^2_S.\label{thm1}
\end{gather}
The stability of the loosely-coupled scheme~\eqref{S2discrete}-\eqref{S1discrete} is stated in the following result.
\begin{theorem}
Let $\{(\velfh^n, p_{h}^n, \disvth^n,  \disvh^n, \displh^n \}_{0 \le n \le N}$ be the solution of~\eqref{S2discrete}-\eqref{S1discrete}.
Then, the following estimate holds:
  \begin{gather}
  \mathcal{E}_f(\velfh^N)+\mathcal{E}_v(\disvh^N)+\mathcal{E}_s(\displh^N)  + \frac{\Delta t^2}{2 \rho_s \epsilon} \n\str(\velfh^{N},p_h^{N}) \norm\n^2_{L^2(\Gamma)} +\frac{\rho_f \Delta t^2}{2}\sum_{n=0}^{N-1}\n  d_t \velfh^{n+1} \n^2_{L^2(\domf)}
     \notag \\  
 +\frac{\Delta t^2}{2}\sum_{n=0}^{N-1}\n d_t\displh^{n+1}\n^2_S
+ \mu \Delta t \sum_{n=0}^{N-1} \n \velfh^{n+1} \n^2_{F}  +\frac{\rho_s \epsilon}{2}\sum_{n=0}^{N-1}\n \disvth^{n+1}-\disvh^{n} \n^2_{L^2(\Gamma)}  
\notag \\
\ltt   \mathcal{E}_f(\velfh^0)+\mathcal{E}_v(\disvh^0)+\mathcal{E}_s(\displh^0)   + \frac{\Delta t^2}{2 \rho_s \epsilon} \n\str(\velfh^{0},p_h^{0}) \norm\n^2_{L^2(\Gamma)}
+\Delta t \sum_{n=0}^{N-1} \n p_{in/out} (t^{n+1})\n^2_{L^2(\Sigma)}.\label{theorem:stability}
  \end{gather}
 \end{theorem}
 \begin{proof}
 To prove the energy estimate, we test the problem~\eqref{S2discrete} with $(\tfh, \tvh, \tph) = (\velfh^{n+1}, \disvh^{n+1}, p_h^{n+1}),$ and problem~\eqref{S1discrete} with
  $\tsh = \disvth^{n+1} = d_t \displh^{n+1}.$
  Then, after adding them together, multiplying by $\Delta t$, and using identity
\begin{equation}
(a-b)a = \frac{1}{2}a^2-\frac{1}{2}b^2+\frac{1}{2}(a-b)^2,
\label{identity}
\end{equation}
     we get
  \begin{gather}
  \frac{\rho_f}{2}\left(\n \velfh^{n+1} \n^2_{L^2(\domf)}-\n \velfh^{n} \n^2_{L^2(\domf)}+ \n \velfh^{n+1}-\velfh^n \n^2_{L^2(\domf)}\right)+ 2 \mu \Delta t\n \du(\velfh^{n+1}) \n^2_{L^2(\domf)}
   +  \frac{\rho_s \epsilon}{2}\left(\n \disvh^{n+1} \n^2_{L^2(\Gamma)}-\n \disvh^{n} \n^2_{L^2(\Gamma)}\right)
    \notag \\
  +  \frac{\rho_s \epsilon}{2}\left( \n \disvth^{n+1}-\disvh^n \n^2_{L^2(\Gamma)}+ \n \disvh^{n+1}-\disvth^{n+1} \n^2_{L^2(\Gamma)}\right)
 +\frac{1}{2} a_s(\displh^{n+1}, \displh^{n+1})-\frac{1}{2} a_s(\displh^{n}, \displh^{n})
  \notag \\
+\frac{1}{2} a_s(\displh^{n+1}-\displh^{n}, \displh^{n+1}-\displh^{n})
=\Delta t  \int_{\Gamma} \str(\velfh^{n},p_h^{n}) \norm \cdot \left( \disvh^{n+1} - \disvth^{n+1}  \right)dS
\notag
+\Delta t\int_{\Sigma} p_{in/out}(t^{n+1}) \velfh^{n+1}\cdot\norm dS.
  \end{gather}
Since term $\frac{\varrho_s}{\Delta t}\disvth^{n+1}$ appears in both equations~\eqref{S1discrete}, \eqref{S2discrete}, but with the opposite sign, we used \eqref{identity} to cancel  the intermediate term $\n \disvth^{n+1} \n^2_{L^2(\Gamma)}$ in the estimate above. Denote by $\mathcal{I} = \Delta t  \int_{\Gamma} \str(\velfh^{n},p_h^{n})\norm \cdot \left( \disvh^{n+1} - \disvth^{n+1}  \right)dS$ the term that corresponds to the splitting error.  From~\eqref{robin1} we have
\begin{align}
&  \disvh^{n+1}-\disvth^{n+1}=\frac{\Delta t}{\rho_s \epsilon}\left( -\str(\velfh^{n+1},p_h^{n+1}) \norm+ \str(\velfh^{n},p_h^{n})\norm \right) \quad  \textrm{on} \; \Gamma. \label{stresses}
\end{align}
Now, we can write $\mathcal{I}$ as
\begin{gather}
\mathcal{I} =\frac{\Delta t^2}{\rho_s \epsilon} \int_{\Gamma} \str(\velfh^{n},p_h^{n}) \norm \cdot \left( \str(\velfh^{n},p_h^{n}) \norm  -\str(\velfh^{n+1},p_h^{n+1}) \norm \right)dS   \notag \\
= \frac{\Delta t^2}{2 \rho_s \epsilon} \left( \n\str(\velfh^{n},p_h^{n}) \norm\n^2_{L^2(\Gamma)}- \n\str(\velfh^{n+1},p_h^{n+1}) \norm\n^2_{L^2(\Gamma)}
 + \n\str(\velfh^{n},p_h^{n}) \norm-\str(\velfh^{n+1},p_h^{n+1}) \norm \n^2_{L^2(\Gamma)}\right)
\notag \\
= \frac{\Delta t^2}{2 \rho_s \epsilon} \left( \n\str(\velfh^{n},p_h^{n}) \norm\n^2_{L^2(\Gamma)}- \n\str(\velfh^{n+1},p_h^{n+1}) \norm\n^2_{L^2(\Gamma)} \right)+\frac{\rho_s \epsilon}{2}\n\disvh^{n+1}-\disvth^{n+1} \n^2_{L^2(\Gamma)}. \label{Iest}
\end{gather}
Employing identity~\eqref{Iest} and summing from $n=0$ to $N-1$, we obtain
  \begin{gather}
  \mathcal{E}_f(\velfh^N)+\mathcal{E}_v(\disvh^N)+\mathcal{E}_s(\displh^N)
  + \frac{\Delta t^2}{2 \rho_s \epsilon} \n\str(\velfh^{N},p_h^{N}) \norm\n^2_{L^2(\Gamma)}
  +\frac{\rho_f \Delta t^2}{2}\sum_{n=0}^{N-1} \n d_t \velfh^{n+1} \n^2_{L^2(\domf)}  
    +\frac{\Delta t^2}{2}\sum_{n=0}^{N-1}\n d_t\displh^{n+1}\n^2_S
  \notag \\
 + 2 \mu \Delta t \sum_{n=0}^{N-1}\n \du(\velfh^{n+1}) \n^2_{L^2(\domf)}    
  +\frac{\rho_s \epsilon}{2}\sum_{n=0}^{N-1}\n \disvth^{n+1}-\disvh^{n} \n^2_{L^2(\Gamma)}
=   \mathcal{E}_f(\velfh^0)+\mathcal{E}_v(\disvh^0)
+\mathcal{E}_s(\displh^0) 
\notag\\
 +  \frac{\Delta t^2}{2 \rho_s \epsilon} \n\str(\velfh^{0},p_h^{0}) \norm\n^2_{L^2(\Gamma)}+ \Delta t \sum_{n=0}^{N-1} \int_{\Sigma} p_{in/out}(t^{n+1}) \velfh^{n+1}\cdot\norm dS.\label{stabn-1}
  \end{gather}
Using the Cauchy-Schwarz, the trace, and the Korn inequalities, we can estimate 
\begin{align} 
&\bigg|\Delta t \int_{\Sigma} p_{in/out}(t^{n+1}) \velfh^{n+1}\cdot\norm dS \bigg| \le \frac{C \Delta t}{4 \mu}\n p_{in/out}(t^{n+1})\n^2_{L^2(\Sigma)}+\mu \Delta t \n \du(\velfh^{n+1}) \n^2_{L^2(\domf)}.
\end{align} 
  Combining the latter estimates with equation~\eqref{stabn-1} we prove the desired energy inequality.
 \end{proof}

\section{Error Analysis}~\label{Sec:Convergence}
In this section, we analyze the convergence rate of the kinematically coupled $\beta$ scheme~\eqref{S1discrete}-\eqref{S2discrete} when $\beta=1$.
We assume that the true solution satisfies the following assumptions: 
\begin{align}
\velf &\in  H^1(0,T; H^{k+1}(\domf))\cap H^2(0,T; L^2(\domf)),\;\; \velf|_{\Gamma} \in H^1(0,T; H^{k+1}(\Gamma)), \label{reg1}\\
p & \in L^2(0,T; H^{s+1}(\domf)),\;\;  p|_{\Gamma} \in H^{1}(0,T; H^{s+1}(\Gamma)),\;\; \\
\displ & \in W^{1,\infty}(0,T; H^{k+1}(\g)) \cap H^2(0,T; H^{k+1}(\g))\cap H^3(0,T; L^2(\g)). \label{reg2}
\end{align}
To approximate the problem in space, we apply the Lagrangian finite elements of polynomial degree $k$ for all the variables, except for the fluid pressure, for which we use elements of degree $s<k$.  
  We assume that our finite element spaces satisfy the usual approximation properties, and that the fluid velocity-pressure spaces  satisfy the discrete \emph{inf-sup} condition. We introduce the following time discrete norms:
 \begin{equation}
  \n\tf\n_{L^2(0,T; X)} = \bigg(\Delta t \sum_{n=0}^{N-1} \n\tf^{n+1}\n^2_{X} \bigg)^{1/2}, \quad \n\tf \n_{L^{\infty}(0,T; X)} = \max_{0 \le n \le N} \n\tf^n \n_{X}, \label{tdiscnorm}
 \end{equation}
where $X \in \{H^k(\domf), H^k(\g),  F, S\}$,  where norm $\|.\|_F$ is defined below equation \eqref{FluidBilinear}, norm $\|.\|_S$ is defined below equation~\eqref{strucinit}.  Note that they are equivalent to the continuous norms since we use piecewise constant approximations in time.

Let $I_h$ be the Lagrangian interpolation operator onto $V^s_h,$ and
$R_h$ be the  Ritz projector onto  $V_h^s$ such that for all $\displ \in V^s$
\begin{equation}
a_s(\displ - R_h \displ, \tsh) = 0  \quad \forall \tsh \in V_h^s. \label{Ritz}
\end{equation}
Then, the finite element theory for Lagrangian and Ritz projections~\cite{ciarlet1978finite} gives, respectively,
\begin{equation}
\n \disv - I_h \disv\n_{L^2(\g)} + h \n \disv - I_h \disv\n_{H^1(\g)} \ltt h^{k+1} \n \disv \n_{H^{k+1}(\g)}, \quad \forall \disv \in V^s, \label{app1}
\end{equation}
and 
\begin{equation}
\n \displ - R_h \displ \n_{S}  \ltt  h^{k} \n \displ \n_{H^{k+1}(\g)}, \quad \forall \displ \in V^s.
\end{equation}
Let $\Pi_h$ be a projection operator onto $Q_h^f$ such that
\begin{equation}
\n p - \Pi_h p \n_{L^2(\domf)} \ltt  h^{s+1} \n p \n_{H^{s+1}(\domf)}, \quad \forall p \in Q^f.
\end{equation}
Following the approach in~\cite{Fernandez3}, we introduce a Stokes-like projection operator $(S_h, P_h): V^f \rightarrow V^f_h \times Q^f_h$, defined for all $\velf \in V^f$ by
\begin{align}
& (S_h \velf, P_h \velf) \in V^f_h \times Q^f_h, \\
& (S_h \velf)|_{\g} = I_h (\velf|_{\g}),  \\
& a_f(S_h \velf, \tfh)-b(P_h \velf, \tfh) = a_f(\velf, \tfh),\quad \forall \tfh \in V^f_h \; \textrm{such that} \; \tfh|_{\g}=0,\\
& b(\tph, S_h \velf) = 0, \quad \forall \tph \in Q^f_h. \label{press_proj}
\end{align} 
Projection operator $S_h$ satisfies the following approximation properties (see~\cite{Fernandez3}, Theorem B.5):
\begin{equation}
\n \velf - S_h \velf\n_{F} \ltt h^k \n \velf \n_{H^{k+1}(\domf)}. \label{app2}
\end{equation}

We assume that the continuous fluid velocity lives in the space $V^{fd}=\{\velf \in V^f | \; \nabla \cdot \velf =0\}$. Since the test functions for the partitioned scheme do not satisfy the kinematic coupling condition, we start by deriving the monolithic variational formulation with the test functions in $V_h^{f}\times V_h^s \times Q_h^f$: Find $(\velf, \displ, p) \in  V^{fd}  \times V^s \times Q^f$ with $\velf^{n+1} = \partial_t \displ^{n+1}$ on $\g$ such that for all $(\tfh, \tsh, \tph) \in V_h^{f}\times V_h^s \times Q_h^f$ we have
\begin{gather}
\rho_f \int_{\domf} \dt \velf^{n+1} \cdot \tfh \dx +a_f(\velf^{n+1},\tfh)-b(p^{n+1}, \tfh) +  \rho_{s} \epsilon\int_{\g} \dtt \displ^{n+1} \cdot  \tsh dS   +a_s (\displ^{n+1}, \tsh) \notag\\
 = \int_{\Sigma} p_{in/out}(t^{n+1}) \tfh \cdot \norm  dS + \int_{\g} \str(\velf^{n+1},p^{n+1}) \norm (\tfh - \tsh) dS. 
\end{gather}
Notice that here the fluid and the structure test functions are independent, i.e. we do not satisfy condition $(\tfh)_{|\Gamma}=\tsh$.
Introducing variables $\disv^{n+1}=\partial_t \displ^{n+1}$ and $\disvt^{n+1} = \velf^{n+1}|_{\g}$, we can rewrite the structure acceleration term as follows
\begin{gather}
 \rho_{s} \epsilon\int_{\g} \dtt \displ^{n+1} \cdot  \tsh dS =  \rho_{s} \epsilon\int_{\g} \dt \disv^{n+1} \cdot  \tsh dS 
 =  \rho_{s} \epsilon\int_{\g} \frac{\disv^{n+1} -\disvt^{n+1}}{\Delta t} \cdot  \tsh dS \notag \\
 +\rho_{s} \epsilon\int_{\g} \frac{\disvt^{n+1} -\disv^{n}}{\Delta t} \cdot  \tsh dS + \rho_{s} \epsilon\int_{\g} (\dt \disv^{n+1}- d_t \disv^{n+1}) \cdot \tsh dS. \label{accmon}
\end{gather}
Taking into account the latter equation, the weak formulation of the monolithic problem can be written as
\begin{gather}
\rho_f \int_{\domf} d_t \velf^{n+1} \cdot \tfh \dx +a_f(\velf^{n+1},\tfh)-b(p^{n+1}, \tfh)  +a_s (\displ^{n+1}, \tsh)+  \rho_{s} \epsilon\int_{\g} \frac{\disvt^{n+1} -\disv^{n}}{\Delta t} \cdot  \tsh dS 
 \notag\\
+\rho_{s} \epsilon\int_{\g} \frac{\disv^{n+1} -\disvt^{n+1}}{\Delta t} \cdot  \tsh dS  
 =\rho_f \int_{\domf} (d_t \velf^{n+1}-\dt \velf^{n+1}) \cdot \tfh \dx
   +\rho_{s} \epsilon\int_{\g} (d_t \disv^{n+1} - \dt \disv^{n+1}) \cdot \tsh dS
   \notag \\
+ \int_{\g} \str(\velf^{n+1},p^{n+1}) \norm \cdot(\tfh - \tsh) dS+ \int_{\Sigma} p_{in/out}(t^{n+1}) \tfh \cdot \norm  dS. \label{weakmonolithicerror}
\end{gather}
To analyze the error of our numerical scheme, we start by subtracting~\eqref{S1discrete}-\eqref{S2discrete} from~\eqref{weakmonolithicerror}, giving rise to the following error equations:
\begin{gather}
\rho_f \int_{\domf} d_t (\velf^{n+1}-\velfh^{n+1}) \cdot \tfh \dx +a_f(\velf^{n+1}-\velfh^{n+1},\tfh)-b(p^{n+1}-p_h^{n+1}, \tfh)-b(\tph, \velfh^{n+1}) 
   +a_s(\displ^{n+1}-\displh^{n+1}, \tsh) 
   \notag\\
+  \rho_{s} \epsilon\int_{\g} \left(\frac{\disvt^{n+1} -\disv^{n}}{\Delta t}-\frac{\disvth^{n+1} -\disvh^{n}}{\Delta t}\right) \cdot  \tsh dS  
 +  \rho_{s} \epsilon\int_{\g} \left(\frac{\disv^{n+1} -\disvt^{n+1}}{\Delta t}-\frac{\disvh^{n+1} -\disvth^{n+1}}{\Delta t}\right) \cdot  \tvh dS
 \notag \\
   - \int_{\g} (\str(\velf^{n},p^{n}) \norm -\str(\velfh^{n},p_h^{n}) \norm) \cdot(\tvh - \tsh) dS
= \mathcal{R}^f(\tfh)+\mathcal{R}^s(\tsh)+\mathcal{R}^{os}(\tvh - \tsh), 
\label{erroreq}
\end{gather}
for all $(\tfh, \tvh, \tsh) \in V_h^f \times V_h^s \times V_h^s $ such that $\tfh|_{\g} = \tvh,$
where
\begin{align}
& \mathcal{R}^f(\tfh)  = \rho_f \int_{\domf} (d_t \velf^{n+1}-\dt \velf^{n+1}) \cdot \tfh \dx  \\
& \mathcal{R}^s(\tsh)  = \rho_{s} \epsilon\int_{\g} (d_t \disv^{n+1} - \dt \disv^{n+1}) \cdot \tsh dS, \\
& \mathcal{R}^{os}(\tvh - \tsh)  = \int_{\g} \left(\str(\velf^{n+1},p^{n+1}) \norm-\str(\velf^{n},p^{n})\norm +\frac{\disv^{n+1} -\disvt^{n+1}}{\Delta t} \right)  \cdot(\tvh - \tsh) dS.
\end{align}
Note that the last term accounts for the operator-splitting error. Since $\disvt^{n+1} = \velf^{n+1}|_{\g} = \dt \displ^{n+1} = \disv^{n+1}$, we have
\begin{gather}
 \mathcal{R}^{os}(\tvh - \tsh)  = \int_{\g} \left(\str(\velf^{n+1},p^{n+1}) \norm-\str(\velf^{n},p^{n})\norm \right)  \cdot(\tvh - \tsh) dS.
\end{gather}
We split the error of the method as a sum of the approximation error $\theta_r^{n+1}$ and the truncation error $\delta_r^{n+1}$, for $r \in \{f,\tilde{v},p,s,v\}$ as follows
\begin{align}
e_f^{n+1} & =\velf^{n+1}-\velfh^{n+1} = (\velf^{n+1}-S_h \velf^{n+1})+(S_h \velf^{n+1}-\velfh^{n+1}) = \thf^{n+1}+\defl^{n+1}, \label{error1}\\
e_{\tilde{v}}^{n+1} & =\disvt^{n+1}-\disvth^{n+1} = (\disvt^{n+1}-I_h \disvt^{n+1})+(I_h \disvt^{n+1}-\disvth^{n+1}) = \thvt^{n+1}+\devt^{n+1},\\
e_p^{n+1} & =p^{n+1}-p_h^{n+1} = (p^{n+1}-\Pi_h p^{n+1})+(\Pi_h p^{n+1}-p_h^{n+1}) = \thp^{n+1}+\dep^{n+1}, \\
e_s^{n+1} & =\displ^{n+1}-\displh^{n+1} = (\displ^{n+1}-R_h \displ^{n+1})+(R_h \displ^{n+1}-\displh^{n+1}) = \ths^{n+1}+\des^{n+1}, \\
e_v^{n+1} & =\disv^{n+1}-\disvh^{n+1} = (\disv^{n+1}-I_h \disv^{n+1})+(I_h \disv^{n+1}-\disvh^{n+1}) = \thv^{n+1}+\dev^{n+1}. \label{error2}
\end{align}
The main result of this section is stated in the following theorem.

\begin{theorem}\label{MainThm}
Consider the solution $(\velfh, p_h, \disvth, \disvh, \displh)$ of~\eqref{S1discrete}-\eqref{S2discrete}, with discrete initial data
$(\velfh^0, p_h^0, \disvth^0, \disvh^0, \displh^0) = (S_h \velf^0, \Pi_h p^0, I_h \disvt^0, I_h \disv^0, R_h \displ^0)$. Assume that $\beta=1$ and the exact solution satisfies assumptions~\eqref{reg1}-\eqref{reg2}.  Furthermore, we assume that
$$\gamma \Delta t <1, \qquad  \gamma_1 < \displaystyle\frac{\rho_s \epsilon}{8 \Delta t}, \qquad \gamma_2<\frac{1}{4},$$
 where $\gamma>0, \gamma_1>0, \gamma_2>0$. Let $\tilde{\gamma}=\max\{\gamma,\gamma_2,\gamma_3\}.$
 Then, the following estimate holds
\begin{gather}
\n \velf^N - \velfh^N\n_{L^2(\domf)}
+ \n \velf^N - \velfh^N\n_{L^2(0,T;F)}
+\n \disv^N- \disvh^N\n_{L^2(\Gamma)}
+\n \displ^N-\displh^N\n_{S}
+\n \str(\velf^N,p^N)\norm-\str(\velfh^N, p_h^N)\norm \n_{L^2(\g)}
\notag \\ 
\ltt  e^{\tilde{\gamma} T} \left(\Delta t \mathcal{A}_1 +\Delta t^2 \Big(\Delta t^{1/2}+\frac{ 1}{ \gamma_2 } + \frac{1}{\gamma_1} +\gamma_1 \Delta t \Big) \mathcal{A}_2 + h^{k} \mathcal{B}_1 +h^{k+1} \mathcal{B}_2 +h^{s+1} \mathcal{B}_3 \right.
\notag \\
\left.
+ \Delta t h^{k} \Big(\Delta t+\frac{ 1}{ \gamma_2 } + \frac{1}{\gamma_1} +\gamma_1 \Delta t^2 \Big) \mathcal{C}_1
+ \Delta t h^{s+1} \Big(\Delta t+\frac{ 1}{ \gamma_2 } + \frac{1}{\gamma_1} +\gamma_1 \Delta t^2 \Big) \mathcal{C}_2, \notag
\right)
  \end{gather}
  where norm $\|.\|_F$ is defined below equation \eqref{FluidBilinear}, norm $\|.\|_S$ is defined below equation~\eqref{strucinit} and
  \begin{align}
  \mathcal{A}_1 &=    \n \dtt \velf \n_{L^2(0,T,L^2(\domf))} +\frac{1}{\gamma} \n \dtt \disv\n_{L^2(0,T;L^2(\Gamma))}+\frac{1}{\gamma} \n \dtt \displ\n_{L^2(0,T; H^1(\Gamma))}+\frac{1}{\gamma_1}  \n \partial_t  \str \norm \n_{L^2(0,T; L^2(\Gamma))}, \notag \\
    \mathcal{A}_2 &=    \n \partial_t  \str \norm \n_{L^2(0,T; L^2(\g))}, \notag \\
  \mathcal{B}_1 &= \frac{1}{\gamma}  \| \disv \|_{L^2(0,T;H^{k+1}(\Gamma))}
+ \n\partial_t \velf \n_{L^2(0,T;H^{k+1}(\domf))}
+ \n \velf \n_{L^2(0,T;H^{k+1}(\domf))}
  +\frac{1}{ \gamma_1 }  \n\velf \n_{L^2(0,T;H^{k+1}(\g))} 
  +\| \velf\|_{L^{\infty}(0,T; H^{k+1}(\domf))}
  \notag \\
 &    +\| \velf\|_{L^{\infty}(0,T; H^{k+1}(\g))}
    +\| \displ \|_{L^{\infty}(0,T; H^{k+1}(\g))},
\qquad    
 \mathcal{B}_2 = \left(1+\frac{1}{\gamma_1} \right)  \n\partial_t \disv\n_{L^2(0,T;H^{k+1}(\Gamma))}  +\| \disv \|_{L^{\infty}(0,T; H^{k+1}(\g))}, 
  \notag \\
 \mathcal{B}_3 &=\n p \n^2_{L^2(0,T; H^{s+1}(\domf))}+\frac{1}{ \gamma_1 } \n p \n^2_{L^2(0,T; H^{s+1}(\g))}  +\|p \|_{L^{\infty}(0,T; H^{s+1}(\g))}, \notag \\
  \mathcal{C}_1 &= \n\partial_t \velf \n^2_{L^2(0,T;H^{k+1}(\g))}, \qquad
   \mathcal{C}_2 =\n\partial_t p \n^2_{L^2(0,T;H^{s+1}(\g))}. \notag
  \end{align}
\end{theorem}
\begin{proof}
Due to property~\eqref{Ritz} of the Ritz projection operator,  we have $a_s(\ths^{n+1}, \tsh) = 0.$ 
Furthermore, since $\disvt^{n+1} = \disv^{n+1}$, we have $ \thvt^{n+1} -\thv^{n} = \thv^{n+1}-\thv^{n}$ 
and
$
\thv^{n+1} -\thvt^{n+1} =0.
$
Rearranging the error equation~\eqref{erroreq} and taking the properties above into account, we get
\begin{gather}
\rho_f \int_{\domf} d_t \defl^{n+1} \cdot \tfh \dx +a_f(\defl^{n+1},\tfh)-b(\dep^{n+1}, \tfh)-b(\tph, \velfh^{n+1})
  +a_s(\des^{n+1}, \tsh)  
+  \rho_{s} \epsilon\int_{\g} \frac{\devt^{n+1} -\dev^{n}}{\Delta t} \cdot  \tsh dS  \notag \\
 +  \rho_{s} \epsilon\int_{\g} \frac{\dev^{n+1} -\devt^{n+1}}{\Delta t} \cdot  \tvh dS
  - \int_{\g} \str(\defl^{n}, \dep^{n}) \norm \cdot(\tvh - \tsh) dS=\mathcal{R}^f(\tfh)+\mathcal{R}^s(\tsh) 
+\mathcal{R}^{os}(\tvh - \tsh) 
  \notag \\
 - \rho_f \int_{\domf} d_t \thf^{n+1} \cdot \tfh \dx
 -a_f(\thf^{n+1},\tfh)+b(\thp^{n+1}, \tfh) 
 -  \rho_{s} \epsilon\int_{\g}\frac{\thvt^{n+1} -\thv^{n}}{\Delta t} \cdot  \tsh dS 
 \notag \\
+\int_{\g} \str(\thf^{n}, \thp^{n}) \norm \cdot(\tvh - \tsh) dS, \label{error_s}
\end{gather}
for all $(\tfh, \tvh, \tsh) \in X_h^f \times V^s_h \times V_h^s$ such that $\tfh|_{\g} = \tvh$.
We proceed by choosing  test functions $\tfh = \defl^{n+1}, \tvh = \dev^{n+1}, \tph = \dep^{n+1}$ and $\tsh=\devt^{n+1}$. Thanks to~\eqref{press_proj}, the pressure terms simplify as follows
\begin{gather}
-b(\dep^{n+1}, \defl^{n+1})-b(\dep^{n+1}, \velfh^{n+1}) = -b(\dep^{n+1}, S_h \velf^{n+1}) = 0.
\end{gather}
Multiplying equation~\eqref{error_s} by $\Delta t$ and summing over $0 \le n \le N-1$, we get
\begin{gather}
\mathcal{E}_f(\defl^N)+\mathcal{E}_v(\dev^N) +\frac{\rho_f \Delta t^2}{2}\sum_{n=0}^{N-1}\n  d_t \defl^{n+1} \n^2_{L^2(\domf)}
+2 \mu \Delta t\sum_{n=0}^{N-1} \n \defl^{n+1} \n^2_{F}  
 +\frac{\rho_s \epsilon}{2}\sum_{n=0}^{N-1}\n \devt^{n+1}-\dev^{n} \n^2_{L^2(\g)}  
   \notag \\  
  +\frac{\rho_s \epsilon}{2}\sum_{n=0}^{N-1}\n \dev^{n+1}-\devt^{n+1} \n^2_{L^2(\g)}  
   +\Delta t \sum_{n=0}^{N-1} a_s(\des^{n+1}, \devt^{n+1})  
  - \Delta t \sum_{n=0}^{N-1} \int_{\g} \str(\defl^{n}, \dep^{n}) \norm \cdot( \dev^{n+1} -\devt^{n+1}) dS
  =   \mathcal{E}_f(\defl^0) +\mathcal{E}_v(\dev^0)
    \notag \\  
  +\Delta t \sum_{n=0}^{N-1} \big(\mathcal{R}^f(\defl^{n+1})+\mathcal{R}^s(\dev^{n+1})+\mathcal{R}^{os}(  \dev^{n+1} - \devt^{n+1})\big)
-\rho_f \Delta t \sum_{n=0}^{N-1} \int_{\domf} d_t \thf^{n+1} \cdot \defl^{n+1} \dx 
 -\Delta t \sum_{n=0}^{N-1}a_f(\thf^{n+1},\defl^{n+1})
  \notag \\ 
  +\Delta t \sum_{n=0}^{N-1}b(\thp^{n+1}, \defl^{n+1}) 
 - \rho_{s} \epsilon \Delta t \sum_{n=0}^{N-1}\int_{\g} d_t \thv^{n+1}  \cdot  \devt^{n+1} dS  
 +\Delta t  \sum_{n=0}^{N-1}\int_{\g} \str(\thf^{n}, \thp^{n}) \norm \cdot( \dev^{n+1} -\devt^{n+1}) dS. \label{error_eng}
  \end{gather}
 For the term $\Delta t \sum_{n=0}^{N-1} a_s(\des^{n+1}, \devt^{n+1}),$ we proceed as follows 
 \begin{gather*}
 \Delta t \sum_{n=0}^{N-1} a_s(\des^{n+1}, \devt^{n+1}) = \Delta t \sum_{n=0}^{N-1} a_s(\des^{n+1}, d_t \des^{n+1}+I_h \disvt^{n+1}-R_h d_t \displ^{n+1})=\mathcal{E}_s(\des^N) -\mathcal{E}_s(\des^0)  
   \notag \\
    +\frac{\Delta t^2}{2}\sum_{n=0}^{N-1}\n d_t\des^{n+1}\n^2_S
 + \Delta t \sum_{n=0}^{N-1} a_s(\des^{n+1}, I_h \disvt^{n+1}-R_h d_t \displ^{n+1}).
 \end{gather*}  
 Note that, since $\disvt^{n+1}=\disv^{n+1}$, $I_h \disvt^{n+1}-R_h d_t \displ^{n+1} = I_h \disv^{n+1}-\disv^{n+1}+\disv^{n+1}-R_h d_t \displ^{n+1} = -\thv^{n+1}+d_t \ths^{n+1}+\dt \displ^{n+1}-d_t \displ^{n+1}.$ Hence, using property~\eqref{Ritz} of the Ritz projection operator, Cauchy-Schwartz and Young's inequalities, we have
 \begin{align}
 \Delta t \sum_{n=0}^{N-1} a_s(\des^{n+1}, I_h \disvt^{n+1}-R_h d_t \displ^{n+1}) \ltt \frac{\Delta t}{\gamma}\sum_{n=0}^{N-1} \| \thv^{n+1} \|^2_S+\frac{\Delta t \gamma}{4}\sum_{n=0}^{N-1} \| \des^{n+1}\|^2_S+\Delta t \sum_{n=0}^{N-1} \mathcal{R}^e(\des^{n+1})
 \end{align}
 for $\gamma>0,$ where $\mathcal{R}^e(\des^{n+1}) = a_s(\des^{n+1}, \dt \displ^{n+1}-d_t \displ^{n+1})$.
 
 To estimate the last term on the left hand side of~\eqref{error_eng}, we note that $ \dev^{n+1} -\devt^{n+1} = -(\disvh^{n+1}-\disvth^{n+1})$. Furthermore, adding and subtracting the continuous velocity and pressure on the right hand side of~\eqref{stresses} , we have
 \begin{align}
&  \dev^{n+1} -\devt^{n+1}= - \frac{\Delta t}{\rho_s \epsilon}\left( -\str(S_h\velf^{n+1} - \defl^{n+1},\Pi_h p^{n+1} - \dep^{n+1}) \norm+ \str(S_h\velf^{n} - \defl^{n},\Pi_h p^{n} - \dep^{n})\norm \right) \quad  \textrm{on} \; \Gamma. \label{trstres}
\end{align}
Employing the identity above, we have
  \begin{gather}
 - \Delta t \sum_{n=0}^{N-1} \int_{\g} \str(\defl^{n}, \dep^{n}) \norm \cdot( \dev^{n+1} -\devt^{n+1}) dS = \underbrace{ - \frac{\Delta t^2}{\rho_s \epsilon} \sum_{n=0}^{N-1} \int_{\g} \str(\defl^{n}, \dep^{n}) \norm \cdot( -\str(\defl^{n+1},\dep^{n+1}) \norm+\str(\defl^{n}, \dep^n) \norm) dS }_{\mathcal{T}_1} 
 \notag \\
  +\underbrace{ \frac{\Delta t^2}{\rho_s \epsilon} \sum_{n=0}^{N-1} \int_{\g} \str(\defl^{n}, \dep^{n}) \norm \cdot \left( -\str\left(S_h (\velf^{n+1}-\velf^{n}), \Pi_h (p^{n+1}-p^n) \right) \norm \right) dS.}_{\mathcal{T}_2.}  
 \end{gather}
To estimate term $\mathcal{T}_1$, we apply identity~\eqref{identity} as follows
\begin{gather}
\mathcal{T}_1= - \frac{\Delta t^2}{\rho_s \epsilon} \sum_{n=0}^{N-1} \int_{\g} \str(\defl^{n}, \dep^{n}) \norm \cdot( -\str(\defl^{n+1},\dep^{n+1}) \norm+\str(\defl^{n}, \dep^n))  \norm dS = 
   \frac{\Delta t^2}{2 \rho_s \epsilon}  \| \str(\defl^{N}, \dep^{N}) \norm \|^2_{L^2(\Gamma)}
   \notag \\
  - \frac{\Delta t^2}{2 \rho_s \epsilon} \| \str(\defl^{0}, \dep^{0}) \norm \|^2_{L^2(\Gamma)}
   -\frac{\Delta t^2}{2 \rho_s \epsilon} \sum_{n=0}^{N-1} \| \str(\defl^{n+1}, \dep^{n+1}) \norm -\str(\defl^{n}, \dep^{n}) \norm \|^2_{L^2(\Gamma)}.
\end{gather}
To estimate the last term in the equation above, we again use identity~\eqref{trstres},\eqref{stresses} and Young's inequality with $\gamma_1>0$ as follows
\begin{gather}
    \frac{\Delta t^2}{2 \rho_s \epsilon} \sum_{n=0}^{N-1} \| \str(\defl^{n+1}, \dep^{n+1}) \norm -\str(\defl^{n}, \dep^{n}) \norm \|^2_{L^2(\Gamma)}
    =
       \frac{\Delta t^2}{2 \rho_s \epsilon} \sum_{n=0}^{N-1}  \| \str\left(S_h (\velf^{n+1}-\velf^{n}), \Pi_h (p^{n+1}-p^n) \norm \right)
       \notag \\
      +\frac{\rho_s \epsilon}{\Delta t} (\devt^{n+1}-\dev^{n+1})   \|^2_{L^2(\Gamma)}
       =       \frac{\Delta t^2}{2 \rho_s \epsilon} \sum_{n=0}^{N-1}  \|\str\left(S_h (\velf^{n+1}-\velf^{n}), \Pi_h (p^{n+1}-p^n) \norm \right) \|^2_{L^2(\Gamma)} 
    \notag   \\
  +   \frac{\rho_s \epsilon}{2} \sum_{n=0}^{N-1}  \|\devt^{n+1}-\dev^{n+1}\|^2_{L^2(\Gamma)}
            + \Delta t \sum_{n=0}^{N-1} \int_{\Gamma} (\devt^{n+1}-\dev^{n+1})\str\left(S_h (\velf^{n+1}-\velf^{n}), \Pi_h (p^{n+1}-p^n) \norm \right)dS
            \notag \\
        \leq    
           \frac{\Delta t^2}{2 \rho_s \epsilon} \sum_{n=0}^{N-1}  \|\str\left(S_h (\velf^{n+1}-\velf^{n}), \Pi_h (p^{n+1}-p^n) \norm \right) \|^2_{L^2(\Gamma)} 
             +   \frac{\rho_s \epsilon}{2} \sum_{n=0}^{N-1}  \|\devt^{n+1}-\dev^{n+1}\|^2_{L^2(\Gamma)}
    \notag   \\
            +\frac{\gamma_1 \Delta t}{8} \sum_{n=0}^{N-1}  \|\devt^{n+1}-\dev^{n+1} \|^2_{L^2(\Gamma)}
                +\frac{2\Delta t}{\gamma_1} \sum_{n=0}^{N-1} \|\str\left(S_h (\velf^{n+1}-\velf^{n}), \Pi_h (p^{n+1}-p^n) \norm \right)\|^2_{L^2(\Gamma)}.
\end{gather}
Finally, we estimate $\mathcal{T}_2$ using the Cauchy-Schwartz inequality and Young's inequality with $\gamma_2>0$ as
\begin{gather}
\frac{\Delta t^2}{\rho_s \epsilon} \sum_{n=0}^{N-1} \int_{\g} \str(\defl^{n}, \dep^{n}) \norm \cdot \left( -\str\left(S_h (\velf^{n+1}-\velf^{n}), \Pi_h (p^{n+1}-p^n) \right) 
\norm \right) dS
\notag \\
\leq
\frac{\gamma_2  \Delta t^3}{2 \rho_s \epsilon} \sum_{n=0}^{N-1} \|\str(\defl^{n}, \dep^{n}) \norm \|^2_{L^2(\Gamma)}
+\frac{  \Delta t}{2 \gamma_2 \rho_s \epsilon} \left\| -\str\left(S_h (\velf^{n+1}-\velf^{n}), \Pi_h (p^{n+1}-p^n) \right) 
\norm \right\|^2_{L^2(\Gamma)}.
\end{gather}
 We bound the right hand side of \eqref{error_eng} as follows. Using Cauchy-Schwartz, Young's, Poincar\'e - Friedrichs, and Korn's inequalities, we have the following:
\begin{gather*}
 -\rho_f \Delta t \sum_{n=0}^{N-1}\int_{\domf} d_t\thf^{n+1} \cdot \defl^{n+1} d\boldsymbol x 
  - \Delta t \sum_{n=0}^{N-1} a_f(\thf^{n +1},\defl^{n+1})
+   \Delta t \sum_{n=0}^{N-1} b(\thp^{n+1}, \defl^{n+1}) \\
  \ltt \frac{ \Delta t \rho_f^2}{\mu} \sum_{n=0}^{N-1} \| d_t  \thf^{n+1}\|^2_{L^2(\domf)} + \Delta t \mu \sum_{n=0}^{N-1} \n  \thf^{n+1}\n^2_{F}+\frac{ \Delta t }{\mu}  \sum_{n=0}^{N-1} \n\thp^{n+1}\n^2_{L^2(\Omega)}+\frac{\mu \Delta t}{2}\sum_{n=0}^{N-1}\n\defl^{n+1} \n^2_{F}. 
\end{gather*}
 Manipulating the next couple of terms and using the Cauchy-Schwartz, Poincar\'e - Friedrichs, Korn's and Young's inequalities with $\gamma_1>0$, we get
 \begin{gather*}
- \rho_{s}\Delta t \epsilon \sum_{n=0}^{N-1} \int_{\g}  d_t \thv^{n+1}  \cdot  \devt^{n+1} dS
+ \Delta t  \sum_{n=0}^{N-1}\int_{\g} \str(\thf^{n}, \thp^{n}) \norm \cdot(\devt^{n+1} - \dev^{n+1}) dx
\\
=\rho_{s}\Delta t \epsilon  \sum_{n=0}^{N-1} \int_{\g}  d_t \thv^{n+1}  \cdot \left( (\dev^{n+1}-\devt^{n+1}) - \dev^{n+1} \right) dS 
+ \Delta t  \sum_{n=0}^{N-1}\int_{\g} \str(\thf^{n}, \thp^{n}) \norm \cdot(\devt^{n+1} - \dev^{n+1}) dx=
\\
 =   -\rho_{s}\Delta t \epsilon  \sum_{n=0}^{N-1} \int_{\g} d_t \thv^{n+1} \cdot  \defl^{n+1}|_{\g} dS
 +\Delta t \sum_{n=0}^{N-1} \left(   \int_{\g} \left( \rho_{s} \epsilon d_t \thv^{n+1}  -  \str(\thf^{n}, \thp^{n}) \norm\right)
 \cdot  (\dev^{n+1}-\devt^{n+1}) dS \right)
 \\
  \ltt 
   \Delta t \left( \frac{ \rho_s^2 \epsilon^2 }{\mu} +\frac{4}{\gamma_1} \right) \sum_{n=0}^{N-1} \n d_t \thv^{n+1} \n^2_{L^2(\g)}
    +\frac{\mu \Delta t}{2} \sum_{n=0}^{N-1} \n\defl^{n+1}\n^2_{F}
    +\frac{4\Delta t}{ \gamma_1}  \sum_{n=0}^{N-1} \|   \str(\thf^{n}, \thp^{n}) \norm  \|^2_{L^2(\g)}  
+\frac{\Delta t \gamma_1}{8}  \sum_{n=0}^{N-1} \|   \dev^{n+1}-\devt^{n+1}  \|^2_{L^2(\g)}.  
\end{gather*}

Combining the estimates above with equation~\eqref{error_eng}  and taking into account the assumption on the initial data, we have 
\begin{gather}
\mathcal{E}_f(\defl^N)+\mathcal{E}_v(\dev^N) +\mathcal{E}_s(\des^N)+   \frac{\Delta t^2}{2 \rho_s \epsilon}  \| \str(\defl^{N}, \dep^{N}) \norm \|^2_{L^2(\Gamma)}  +\frac{\rho_f \Delta t^2}{2}\sum_{n=0}^{N-1}\n  d_t \defl^{n+1} \n^2_{L^2(\domf)}
 +\frac{\rho_s \epsilon}{2}\sum_{n=0}^{N-1}\n \devt^{n+1}-\dev^{n} \n^2_{L^2(\g)}  
   \notag \\  
   + \mu \Delta t\sum_{n=0}^{N-1} \n \defl^{n+1} \n^2_{F}  
   +\frac{\Delta t^2}{2}\sum_{n=0}^{N-1}\n d_t\des^{n+1}\n^2_S
 \ltt  \Delta t \sum_{n=0}^{N-1} \big(\mathcal{R}^f(\defl^{n+1})+\mathcal{R}^s(\dev^{n+1})+\mathcal{R}^{os}(  \dev^{n+1} - \devt^{n+1})+\mathcal{R}^e(\des^{n+1})\big)
    \notag \\  
+ \frac{\Delta t}{\gamma} \sum_{n=0}^{N-1}\| \thv^{n+1} \|^2_S
+\frac{ \Delta t \rho_f^2}{\mu} \sum_{n=0}^{N-1} \| d_t  \thf^{n+1}\|^2_{L^2(\domf)}
+   \Delta t \left( \frac{ \rho_s^2 \epsilon^2 }{\mu} +\frac{4}{\gamma_1} \right) \sum_{n=0}^{N-1} \n d_t \thv^{n+1} \n^2_{L^2(\g)}
  + \Delta t \mu \sum_{n=0}^{N-1} \n  \thf^{n+1}\n^2_{F}
 \notag \\
      +\frac{4 \Delta t}{ \gamma_1}  \sum_{n=0}^{N-1} \|   \str(\thf^{n}, \thp^{n}) \norm  \|^2_{L^2(\g)}  
+   \left(\frac{\Delta t^2}{2 \rho_s \epsilon} + \frac{\Delta t}{2\gamma_1} +\frac{  \Delta t}{2 \gamma_2 \rho_s \epsilon} \right)\sum_{n=0}^{N-1}  \|\str\left(S_h (\velf^{n+1}-\velf^{n}), \Pi_h (p^{n+1}-p^n)  \right) \norm \|^2_{L^2(\Gamma)}
\notag \\
   +\frac{ \Delta t }{\mu}  \sum_{n=0}^{N-1} \n\thp^{n+1}\n^2_{L^2(\Omega)}
               +\frac{\gamma_1 \Delta t}{4} \sum_{n=0}^{N-1}  \|\devt^{n+1}-\dev^{n+1} \|^2_{L^2(\Gamma)}
               + \frac{\gamma_2 \Delta t^3}{2 \rho_s \epsilon} \sum_{n=0}^{N-1} \|\str(\defl^{n}, \dep^{n}) \norm \|^2_{L^2(\Gamma)}
+\frac{\Delta t \gamma}{4}\sum_{n=0}^{N-1} \| \des^{n+1}\|^2_S.  
  \end{gather}  
To estimate the approximation and consistency errors, we use Lemmas~\ref{lemma_interpolation} and ~\ref{cons1}, leading to the following inequality
\begin{gather}
\mathcal{E}_f(\defl^N)+\mathcal{E}_v(\dev^N) +\mathcal{E}_s(\des^N)+   \frac{\Delta t^2}{2 \rho_s \epsilon}  \| \str(\defl^{N}, \dep^{N}) \norm \|^2_{L^2(\Gamma)}  +\frac{\rho_f \Delta t^2}{2}\sum_{n=0}^{N-1}\n  d_t \defl^{n+1} \n^2_{L^2(\domf)}
 +\frac{\rho_s \epsilon}{2}\sum_{n=0}^{N-1}\n \devt^{n+1}-\dev^{n} \n^2_{L^2(\g)}  
   \notag \\  
   + \frac{\mu \Delta t}{2} \sum_{n=0}^{N-1} \n \defl^{n+1} \n^2_{F}  
      +\frac{\Delta t^2}{2}\sum_{n=0}^{N-1}\n d_t\des^{n+1}\n^2_S
 \ltt  
  \Delta t^2 \left( \n \dtt \velf \n^2_{L^2(0,T,L^2(\domf))} +\frac{1}{\gamma} \n \dtt \disv\n^2_{L^2(0,T;L^2(\Gamma))} +\frac{1}{\gamma} \n \dtt \displ\n^2_{L^2(0,T; H^1(\Gamma))} \right.
 \notag \\
 \left. +\frac{1}{\gamma_1}  \n \partial_t  \str \norm \n^2_{L^2(0,T; L^2(\Gamma))}\right)
+  \left(1+\frac{1}{\gamma_1} \right)  h^{2k+2} \n\partial_t \disv\n^2_{L^2(0,T;H^{k+1}(\Gamma))}
+   h^{2s+2} \left( \n p \n^2_{L^2(0,T; H^{s+1}(\domf))}+\frac{1}{ \gamma_1 } \n p \n^2_{L^2(0,T; H^{s+1}(\g))} \right)
\notag \\
 +  h^{2k} \left( \frac{1}{\gamma}  \| \disv \|^2_{L^2(0,T;H^{k+1}(\Gamma))}
+ \n\partial_t \velf \n^2_{L^2(0,T;H^{k+1}(\domf))}
+ \n \velf \n^2_{L^2(0,T;H^{k+1}(\domf))}
  +\frac{1}{ \gamma_1 }  \n\velf \n^2_{L^2(0,T;H^{k+1}(\g))}  \right)
\notag \\
+   \left(\frac{\Delta t^2}{2 \rho_s \epsilon} + \frac{\Delta t}{2\gamma_1} +\frac{  \Delta t}{2 \gamma_2 \rho_s \epsilon} \right)\sum_{n=0}^{N-1}  \|\str\left(S_h (\velf^{n+1}-\velf^{n}), \Pi_h (p^{n+1}-p^n) \norm \right) \|^2_{L^2(\Gamma)}
               + \frac{\gamma_2 \Delta t^3}{2 \rho_s \epsilon} \sum_{n=0}^{N-1} \|\str(\defl^{n}, \dep^{n}) \norm \|^2_{L^2(\Gamma)}
\notag \\
                +\frac{\gamma_1 \Delta t}{4} \sum_{n=0}^{N-1}  \|\devt^{n+1}-\dev^{n+1} \|^2_{L^2(\Gamma)}  
+\frac{\Delta t \gamma}{2}\sum_{n=0}^{N-1} \| \des^{n+1}\|^2_S
+ \gamma \Delta t \frac{\rho_s \epsilon}{2} \sum_{n=0}^{N-1} \n\dev^{n+1}\n^2_{L^2(\Gamma)}.
  \end{gather}
We estimate term $\frac{\gamma_1 \Delta t}{4} \sum_{n=0}^{N-1}  \|\devt^{n+1}-\dev^{n+1} \|^2_{L^2(\Gamma)}$ using equation~\eqref{trstres} as follows
\begin{gather}
\frac{\gamma_1 \Delta t}{4} \sum_{n=0}^{N-1}  \|\devt^{n+1}-\dev^{n+1} \|^2_{L^2(\Gamma)} 
\leq \frac{\gamma_1 \Delta t^3}{2 \rho_s^2 \epsilon^2} \sum_{n=0}^{N-1}  \|  \str\left(S_h (\velf^{n+1}-\velf^{n}), \Pi_h (p^{n+1}-p^n)\right) \norm \|^2_{L^2(\Gamma)} 
\notag \\
+\frac{\gamma_1 \Delta t^3}{2 \rho_s^2 \epsilon^2} \sum_{n=0}^{N-1}  \|\str( \defl^{n+1}- \defl^{n}, \dep^{n+1} - \dep^{n})\norm  \|^2_{L^2(\Gamma)} 
\leq \frac{\gamma_1 \Delta t^3}{2 \rho_s^2 \epsilon^2} \sum_{n=0}^{N-1}  \|  \str\left(S_h (\velf^{n+1}-\velf^{n}), \Pi_h (p^{n+1}-p^n)\right) \norm \|^2_{L^2(\Gamma)} 
\notag \\
+\frac{\gamma_1 \Delta t^3}{\rho_s^2 \epsilon^2}   \|\str( \defl^{0}, \dep^{0})\norm  \|^2_{L^2(\Gamma)} 
+\frac{2\gamma_1 \Delta t^3}{\rho_s^2 \epsilon^2} \sum_{n=0}^{N-1}  \|\str( \defl^{n}, \dep^{n} )\norm  \|^2_{L^2(\Gamma)} 
+\frac{\gamma_1 \Delta t^3}{\rho_s^2 \epsilon^2} \|\str( \defl^{N}, \dep^{N} )\norm  \|^2_{L^2(\Gamma)}. 
\end{gather}
Finally, adding and subtracting the continuous solution, and applying Lemmas~\ref{lemma_interpolation} and~\ref{consistency}, we have
\begin{gather}
\Delta t \sum_{n=0}^{N-1}  \|  \str\left(S_h (\velf^{n+1}-\velf^{n}), \Pi_h (p^{n+1}-p^n)\right) \norm \|^2_{L^2(\Gamma)} \leq 2\Delta t \sum_{n=0}^{N-1}  \| \Delta t \str\left(d_t \thf^{n+1}, d_t \thp^{n+1} \right) \norm \|^2_{L^2(\Gamma)} 
\notag \\
+2\Delta t \sum_{n=0}^{N-1}  \|  \str\left(\velf^{n+1}-\velf^{n}, p^{n+1}-p^n \right) \norm \|^2_{L^2(\Gamma)}
\ltt \Delta t^2 h^{2k} \n\partial_t \velf \n^2_{L^2(0,T;H^{k+1}(\g))}+ \Delta t^2 h^{2s+2} \n\partial_t p \n^2_{L^2(0,T;H^{s+1}(\g))}
\notag \\
+\Delta t^2  \n \partial_t  \str \norm \n^2_{L^2(0,T; L^2(\g))}.
\end{gather}  
Assuming that $\gamma \Delta t <1, \gamma_1 < \displaystyle\frac{\rho_s \epsilon}{8 \Delta t}, \gamma_2<\frac{1}{4}$, and applying  the discrete Gronwall inequality~\cite{thomee2006galerkin}, we get
\begin{gather}
\mathcal{E}_f(\defl^N)+\mathcal{E}_v(\dev^N) +\mathcal{E}_s(\des^N)+   \frac{3\Delta t^2}{8 \rho_s \epsilon}  \| \str(\defl^{N}, \dep^{N}) \norm \|^2_{L^2(\Gamma)}  +\frac{\rho_f \Delta t^2}{2}\sum_{n=0}^{N-1}\n  d_t \defl^{n+1} \n^2_{L^2(\domf)}
 +\frac{\rho_s \epsilon}{2}\sum_{n=0}^{N-1}\n \devt^{n+1}-\dev^{n} \n^2_{L^2(\g)}  
   \notag \\  
   + \frac{\mu \Delta t}{2} \sum_{n=0}^{N-1} \n \defl^{n+1} \n^2_{F}  
      +\frac{\Delta t^2}{2}\sum_{n=0}^{N-1}\n d_t\des^{n+1}\n^2_S
 \ltt   e^{\tilde{\gamma} T}\bigg\{
  \Delta t^2 \left( \n \dtt \velf \n^2_{L^2(0,T,L^2(\domf))} +\frac{1}{\gamma} \n \dtt \disv\n^2_{L^2(0,T;L^2(\Gamma))}  \right.
 \notag \\
 \left. +\frac{1}{\gamma} \n \dtt \displ\n^2_{L^2(0,T; H^1(\Gamma))}+\frac{1}{\gamma_1}  \n \partial_t  \str \norm \n^2_{L^2(0,T; L^2(\Gamma))}\right)
 +   \Delta t^4 \left(\Delta t+\frac{ 1}{ \gamma_2 } + \frac{1}{\gamma_1} +\gamma_1 \Delta t^2 \right)
\n \partial_t  \str \norm \n^2_{L^2(0,T; L^2(\g))}
\notag \\
+  \left(1+\frac{1}{\gamma_1} \right)  h^{2k+2} \n\partial_t \disv\n^2_{L^2(0,T;H^{k+1}(\Gamma))}
+   h^{2s+2} \left( \n p \n^2_{l^2(0,T; H^{s+1}(\domf))}+\frac{1}{ \gamma_1 } \n p \n^2_{L^2(0,T; H^{s+1}(\g))} \right)
\notag \\
 +  h^{2k} \left( \frac{1}{\gamma}  \| \disv \|^2_{L^2(0,T;H^{k+1}(\Gamma))}
+ \n\partial_t \velf \n^2_{L^2(0,T;H^{k+1}(\domf))}
+ \n \velf \n^2_{L^2(0,T;H^{k+1}(\domf))}
  +\frac{1}{ \gamma_1 }  \n\velf \n^2_{L^2(0,T;H^{k+1}(\g))}  \right)
\notag \\
+   \Delta t^2 \left(\Delta t+\frac{ 1}{ \gamma_2 } + \frac{1}{\gamma_1} +\gamma_1 \Delta t^2 \right)
\left( h^{2k} \n\partial_t \velf \n^2_{L^2(0,T;H^{k+1}(\g))}+  h^{2s+2} \n\partial_t p \n^2_{L^2(0,T;H^{s+1}(\g))} \right) 
\bigg\}.
  \end{gather}
Recall that the error between the exact and the discrete solution is the sum of the approximation error and the truncation error~\eqref{error1}-\eqref{error2}. Thus, using the triangle inequality and approximation properties~\eqref{app1}-\eqref{app2}, we prove the desired estimate.

\end{proof}

\begin{lemma} \label{cons1}
The following estimate holds for $\gamma>0$:
\begin{gather*}
\Delta t \sum_{n=0}^{N-1} \big(\mathcal{R}^f(\defl^{n+1})+\mathcal{R}^s(\dev^{n+1})+\mathcal{R}^{os}( \dev^{n+1}- \devt^{n+1})+\mathcal{R}^e(\des^{n+1})\big)  \\
\ltt \Delta t^2 \left( \n \dtt \velf \n^2_{L^2(0,T,L^2(\domf))} +\frac{1}{\gamma} \n \dtt \disv\n^2_{L^2(0,T;L^2(\Gamma))} +\frac{1}{\gamma} \n \dtt \displ\n^2_{L^2(0,T; H^1(\Gamma))} +\frac{1}{\gamma_1}  \n \partial_t  \str \norm \n^2_{L^2(0,T; L^2(\Gamma))} \right)\\
+ \frac{\mu \Delta t}{2} \sum_{n=0}^{N-1} \n\defl^{n+1}\n^2_{F}+ \frac{\Delta t \gamma_1}{8} \sum_{n=0}^{N-1} \n \dev^{n+1}-\devt^{n+1} \n^2_{L^2(\Gamma)}
+ \gamma \Delta t \frac{\rho_s \epsilon}{2} \sum_{n=0}^{N-1} \n\dev^{n+1}\n^2_{L^2(\Gamma)}+\frac{\gamma \Delta t}{4} \sum_{n=0}^{N-1} \n\des^{n+1}\n^2_S,
\end{gather*}
\end{lemma}
\begin{proof}
Using Cauchy-Schwartz, Young's, Poincar\'e - Friedrichs, and Korn's inequalities, we have
$$ \Delta t \displaystyle\sum_{n=0}^{N-1}\mathcal{R}^{f}(\defl^{n+1}) \ltt \frac{\Delta t \rho_f^2 }{ \mu} \sum_{n=0}^{N-1}\n d_t \velf^{n+1}-\partial_t \velf^{n+1}\n^2_{L^2(\domf)} + \frac{\mu \Delta t}{2} \sum_{n=0}^{N-1} \n\defl^{n+1}\n^2_{F}.$$
Furthermore, using Cauchy-Schwartz and Young's inequalities, for $\gamma>0$, we have
\begin{gather*}
\Delta t \displaystyle\sum_{n=0}^{N-1}\left(\mathcal{R}^s(\dev^{n+1})+\mathcal{R}^e(\des^{n+1}) \right)\le \frac{\Delta t \rho_s \epsilon}{2 \gamma}\sum_{n=0}^{N-1}\n d_t \disv^{n+1}-\partial_t \disv^{n+1}\n^2_{L^2(\g)} + \gamma \Delta t \frac{\rho_s \epsilon}{2}\sum_{n=0}^{N-1} \n\dev^{n+1}\n^2_{L^2(\g)}
\\
+\frac{\Delta t}{\gamma} \sum_{n=0}^{N-1} \n d_t \displ^{n+1}-\partial_t \displ^{n+1}\n^2_S +\frac{\gamma \Delta t}{4} \sum_{n=0}^{N-1} \n \des^{n+1}\n^2_S.
\end{gather*}
Finally, to estimate the operator splitting error, we use the Cauchy-Schwartz and Young's inequality with $\gamma_1>0$ as follows
\begin{gather*}
\Delta t \displaystyle\sum_{n=0}^{N-1} \mathcal{R}^{os}(\dev^{n+1} - \devt^{n+1} )=\Delta t \displaystyle\sum_{n=0}^{N-1} \int_{\g} \left( \str(\velf^{n+1},p^{n+1}) \norm-\str(\velf^n,p^n) \norm  \right) \cdot  (\devt^{n+1} - \dev^{n+1}) dx
\notag \\
\leq \frac{2\Delta t}{\gamma_1} \displaystyle\sum_{n=0}^{N-1} \n \str(\velf^{n+1},p^{n+1}) \norm-\str(\velf^n,p^n) \norm  \n^2_{L^2(\g)} + \frac{\Delta t \gamma_1}{8} \sum_{n=0}^{N-1} \n \dev^{n+1}-\devt^{n+1} \n^2_{L^2(\Gamma)}.
\end{gather*}
The final estimate follows by applying Lemma~\ref{consistency}.
\end{proof}

\begin{lemma}[Consistency errors] \label{consistency}
The following inequalities hold:
\begin{align*}
&\Delta t \sum_{n=0}^{N-1}\n d_t \tf^{n+1}-\partial_t \tf^{n+1}\n^2_{L^2(S)} \ltt \Delta t^2 \n\partial_{tt} \tf\n^2_{L^2(0,T;L^2(S))}, \\
&\Delta t \displaystyle\sum_{n=0}^{N-1}\n \str^{n+1} \norm-\str^{n} \norm  \n^2_{L^2(\Gamma)}\ltt \Delta t^2  \n \partial_t  \str \norm \n^2_{L^2(0,T; L^2(\Gamma))}.
\end{align*}
\end{lemma}
\begin{proof}
Using the Cauchy-Schwartz inequality, we have
\begin{gather}
\Delta t \sum_{n=0}^{N-1}\n d_t \tf^{n+1}-\partial_t \tf^{n+1}\n^2_{L^2(S)} = \Delta t  \sum_{n=0}^{N-1}  \int_{S} \bigg|\frac{1}{\Delta t} \int_{t^n}^{t^{n+1}} (t-t^n)\partial_{tt} \tf (t) dt \bigg|^2 dx  \notag
\\
 \le \frac{1}{\Delta t} \int_{S} \sum_{n=0}^{N-1} \bigg( \int_{t^n}^{t^{n+1}} |t-t^n|^2 dt \int_{t^n}^{t^{n+1}} |\partial_{tt}  \tf|^2 dt \bigg) dx \le \frac{\Delta t^2}{3} \int_{S} \int_0^T |\partial_{tt} \tf|^2 dt dx
  \ltt \Delta t^2  \n \dtt \tf\n^2_{L^2(0,T; L^2(S))}. \label{constt1}
\end{gather}
To prove the second inequality, we use the Taylor expansion with integral reminder
$$ \str^{n+1} \norm-\str^{n} \norm  = \int_{t^n}^{t^{n+1}} \partial_t  \str \norm dt.$$
Now we have
\begin{gather*}
\Delta t \displaystyle\sum_{n=0}^{N-1}\n \str^{n+1} \norm-\str^{n} \norm  \n^2_{L^2(\Gamma)} =  \Delta t  \sum_{n=0}^{N-1}  \int_{0}^L \bigg|  \int_{t^n}^{t^{n+1}} \partial_t  \str \norm dt \bigg|^2 dx  \notag \\
 \le \Delta t \int_0^L \sum_{n=0}^{N-1} \bigg( \int_{t^n}^{t^{n+1}}  dt \int_{t^n}^{t^{n+1}} |\partial_t  \str \norm|^2 dt \bigg) dx \le \Delta t^2 \int_{0}^L \int_0^T |\partial_t  \str \norm|^2 dt dx \ltt \Delta t^2  \n \partial_t  \str \norm \n^2_{L^2(0,T; L^2(\Gamma))}.
\end{gather*}
The last line in the proof follows from~\eqref{constt1} and the definition of the time discrete norms~\eqref{tdiscnorm}.
\end{proof}

\begin{lemma}[Interpolation errors] \label{lemma_interpolation} The following inequalities hold:
\begin{gather*}
\Delta t \sum_{n=0}^{N-1} \n d_t \thf^{n+1}\n^2_{L^2(\domf)} \le  \n \dt \thf\n^2_{L^2(0,T;L^2(\domf))} \ltt h^{2k} \n\partial_t \velf
\n^2_{L^2(0,T;H^{k+1}(\domf))}, \\
\Delta t \sum_{n=0}^{N-1} \n d_t \thv^{n+1}\n^2_{L^2(\Gamma)} \le  \n \dt \thv\n^2_{L^2(0,T;L^2(\Gamma))} \ltt  h^{2k+2} \n\partial_t \disv\n^2_{L^2(0,T;H^{k+1}(\Gamma))}, \\
\Delta t \sum_{n=0}^{N-1} \n\thf^{n+1}\n^2_{F} \ltt \Delta t \sum_{n=0}^{N-1} h^{2k} \n \velf^{n+1}\n^2_{H^{k+1}(\domf)}  \ltt  h^{2k} \n \velf\n^2_{L^2(0,T;H^{k+1}(\domf))}, \\
\Delta t \sum_{n=0}^{N-1} \n \thv^{n+1} \n^2_S \ltt 
 h^{2k} \n\disv\n^2_{L^2(0,T;H^{k+1}(\Gamma))}, \qquad
\Delta t \sum_{n=0}^{N-1} \n\thp^{n+1}\n^2_{L^2(\domf)} \ltt h^{2s+2} \n p \n^2_{L^2(0,T; H^{s+1}(0,T))}. 
\end{gather*}
\end{lemma}
\begin{proof}
The last three inequalities follow directly from approximation properties~\eqref{app1}-\eqref{app2}. To prove the first equality, we use the Cauchy-Schwarz inequality as follows
  \begin{gather}
 \Delta t \sum_{n=0}^{N-1} \n d_t \thf^{n+1}\n^2_{L^2(\domf)} 
= \frac{1}{\Delta t} \sum_{n=0}^{N-1} \bigg\n \int_{t^n}^{t^{n+1}} \dt \thf^{n+1}\bigg\n^2_{L^2(\domf)} \le  \frac{1}{\Delta t} \sum_{n=0}^{N-1} \int_{\domf} \left(\Delta t \int_{t^n}^{t^{n+1}} |\dt \thf^{n+1}|^2 dt \right) \dx \notag \\
\le \n \dt \thf\n^2_{L^2(0,T;L^2(\domf))} \ltt h^{2k} \n\partial_t \velf \n^2_{L^2(0,T;H^{k+1}(\domf))}.
 \end{gather} 
 The second inequality is proved in an analogous way.
\end{proof}
\if 1=0
\begin{lemma}[Discrete Gronwall's inequality~\cite{thomee2006galerkin}]~\label{gronwall}
Let $(k_n)_{n \ge 0}, (p_n)_{n \ge 0}$ be two non-negative sequences, and $g_0 \ge 0$. Assume that the sequence $(\phi_n)_{n \ge 0}$ satisfies
\begin{align}
& \phi_0 \le g_0, \qquad \phi_n \le g_0+\sum_{s=0}^{n-1} p_s + \sum_{s=0}^{n-1} k_s \phi_s, \quad \forall n \ge 1.
\end{align} 
Then,
\begin{equation}
\phi_n \le \left( g_0 + \sum_{s=0}^{n-1} p_s\right) \exp \left(\sum_{s=0}^{n-1}k_s \right), \quad \forall n \ge 1.
\end{equation}
\end{lemma}
\fi

\section{Thick structure models and other extensions}\label{Sec:thick}
One of the most appealing features of the kinematically coupled $\beta$-scheme and its variants is that it can be generalized to the various FSI problems including the ones with  thick structures \cite{thick} and composite structures \cite{multilayered}. The stability and the convergence proof presented in this paper can  be applied, with simple and straightforward modifications, to the $\beta$-scheme for the fluid-composite structure interaction problem~\cite{multilayered}, where the composite structure consists of a thin layer and a thick layer. The main reason for that is the fact that  the fluid and  thick structure are coupled via the thin elastic interface which regularizes the problem (this regularization is quantified in $1D$ case in~\cite{BorisSimplifiedFSI}). It was proven in~\cite{SunBorMulti} that classical kinematically coupled scheme (case $\beta=0$) applied to fluid-composite structure interaction problem is convergent, but the order of convergence is $\mathcal{O}(\Delta t^{1/2})$ in time. Using the methods presented in this paper, one can show that the proposed  modified $\beta$ scheme  applied to a fluid-composite structure interaction problem with $\beta=1$ has optimal, first order in time, convergence.

We will briefly discuss the case of fluid-thick structure interaction problem which is more difficult (numerically and analytically) because there is no additional regularization due to the elastic interface. Thus, only a limited amount of numerical analysis of  partitioned schemes for FSI problems with thick structures is available in the literature.
The generalized Robin-Neumann explicit coupling scheme for the fluid-thick structure interaction problem was analyzed in \cite{fernandez2015convergence} where it was proved that it is convergent, with the order of convergence of $\mathcal{O}(\frac{\Delta t}{\sqrt{h}})$. 
 We consider the $\beta$-scheme for the fluid-thick structure interaction problem presented in~\cite{thick}. A basic stability estimate for the case $\beta=0$ is derived in~\cite{thick} where convergence of the $\beta$-scheme was proved numerically. Here we consider the case when there is no additional structural viscosity (in notation of \cite{thick} case $\epsilon=0$), which analytically and numerically is the most difficult case.  We will show that $\beta$-scheme for FSI with thick structures is  stable under condition $\Delta t^2\ltt h$. The obtained estimates could then be used to prove that the scheme is also convergent with order of accuracy $\mathcal{O}(\frac{\Delta t}{\sqrt{h}})$. Our proof illustrates that numerically our interface has a mass, which makes the scheme convergent.

In the following we  consider a simplified linear version of the fluid-thick structure interaction problem presented in~\cite{thick}. We  start with the weak formulation for the coupled problem and sub-problems of the partitioned scheme. Differential formulation and more details can be found in \cite{thick}. Furthermore, we  ignore the influence of the boundary conditions since they can be treated in the same manner as in the thin structure case.

Let $\Omega_F=(0,1)^2\times (-1,0)$, $\Omega_S=(0,1)^2\times (0,1)$ and $\Gamma=(0,1)^2\times\{0\}$ be the fluid domain, the structure domain and the fluid-structure interface, respectively. We define the corresponding function spaces:
$$
V^f=H^1(\Omega_F)^3,\; Q=L^2(\Omega_F), \;V^s=H^1(\Omega_S)^3, \;
V^{fsi}=\{(\tf, \ts) \in V^f \times V^s | \ \tf|_{\g}  = \ts|_{\g}\}.
$$
Furthermore we introduce a bilinear form connected to the linearized elastic operator:
$$
a_{ts}(\displ,\ts)=\int_{\Omega_S}{\bf S}(\displ):\nabla\ts,
$$
where ${\bf S}(\displ)=2\mu_s {\bf D}(\displ)+\lambda_s(\nabla\cdot\displ){\bf I}$ is the first Piola-Kirchhoff stress tensor and $\mu_s$ and $\lambda_s$ are the Lam\' e parameters. The variational formulation for the coupled fluid-thick structure interaction problem now reads:
\\
Given $t \in (0, T)$ find $(\velf,\disv)\in V^{fsi}$, $p\in Q$, $\displ\in V^s$ such that $\partial_t\displ=\disv$ on $\g$ and for every $(\tf,\ts,q)\in V^{fsi}\times Q$ the following holds:
\begin{gather}
\rho_f \int_{\Omega_F} \dt \velf \cdot \tf \dx +a_f(\velf,\tf)-b(p, \tf)+b(\tp, \velf) +  \rho_{s} \int_{\Omega_S} \dt \disv\cdot \ts \dx 
+a_{ts}(\displ, \ts) =  \int_{\Sigma} p_{in/out}(t) \tf\cdot \norm dS. \label{weakmonolithicthick}
\end{gather}

To discretize the problem in space we use the FEM triangulation with maximum triangle diameter $h$ and define the finite element spaces $V^f_h\subset V^f$, $V^s_h\subset V^s$, and $Q_h\subset Q$.
We denote by $\Omega_F^h$, $\Omega_S^h$, the strip in the fluid and the structure domain, respectively, that consists of all the elements that have at least one vertex on the interface. The width of $\Omega_F^h$ and $\Omega_S^h$ is of order $h$. 

The partitioned numerical scheme for the interaction between a fluid and a thick structure presented in~\cite{thick}, based on the kinematically coupled scheme, reads as follows

\noindent
\textbf{Step 1.} Find $(\velfth^{n+1},\disvth^{n+1})\in V^{fsi}_h$, $\displh^{n+1}\in V^s_h$ such that for every $(\tfh,\tsh)\in V^{fsi}_h$ the following equality holds: 
\begin{equation}\label{A1BisThick}
\begin{array}{c}
\displaystyle{\rho_s\int_{\Omega_S}\frac{\disvth^{n+1}-\disvh^{n}}{\Delta t}\cdot\tsh+a_{ts}(\displh^{n+1},\tsh)+\rho_f\int_{\Omega_F}\frac{\velfth^{n+1}-\velfh^n}{\Delta t}\cdot\tfh=-\int_{\Gamma}\str^n_h\norm\cdot\ts},
\\ \\
\displaystyle{\disvth^{n+1}=\frac{\displh^{n+1}-\displh^{n}}{\Delta t},\quad (\disvth^{n+1})|_{\Gamma}=(\velfth^{n+1})|_{\Gamma}, }
\end{array}
\end{equation}
where $\str^n_h=\str(\velf^n_h,p^{n}_h)$. We emphasize that here and throughout this Section we use $\beta=1$.

Let $\phi_h$ be a  test function such that $(\phi_h)_{|\Gamma}=0$. Then $(\phi_h,0)$ is an admissible test function, so we can take this test function in~\eqref{A1BisThick} to obtain:
$$
\rho_f\int_{\Omega_F}\frac{\velfth^{n+1}-\velfh^n}{\Delta t}\cdot\tfh=0,\quad (\tfh)_{|\Gamma}=0.
$$
Therefore we have $\velfth^{n+1}=\velfh^n$ on all the nodes inside the fluid domain. Hence,  the integral $\int_{\Omega_F}\frac{\velfth^{n+1}-\velfh^n}{\Delta t}\cdot\tfh$ "lives" only in narrow strip $\Omega_F^h$ i.e.
$$
\rho_f\int_{\Omega_F}\frac{\velfth^{n+1}-\velfh^n}{\Delta t}\cdot\tfh=\rho_f\int_{\Omega_F^h}\frac{\velfth^{n+1}-\velfh^n}{\Delta t}\cdot\tfh
.$$
Now, if we take into the account $(\disvth^{n+1})|_{\Gamma}=(\velfth^{n+1})|_{\Gamma}$ we see that~\eqref{A1BisThick} is indeed essentially a  structure problem because the only unknowns are the structure displacement and the structure velocity. However, the fluid inertia enters the problem through the mass matrix on the interface.

\noindent
\textbf{Step 2.} Find $(\velfh^{n+1},\disvh^{n+1},p^{n+1})\in V^{fsi}_h\times Q_h$, such that for $(\tfh,\tsh,q_h)\in V^{fsi}_h\times Q_h$, the following equality holds:
\begin{equation}\label{A2BisThick}
\begin{array}{c}
\displaystyle{
\rho_f\int_{\Omega_F}\frac{\velfh^{n+1}-\velfth^{n+1}}{\Delta t}\cdot\tfh+\rho_s\int_{\Omega_S}\frac{\disvh^{n+1}-\disvth^{n+1}}{\Delta t}\cdot\tsh+a_f(\velfh^{n+1},\tfh)}
\displaystyle{-b(p^{n+1}_h, \tfh)+b(q_h, \velfh^{n+1})=\int_{\Gamma}\str^n_h\norm\cdot\tfh,}
\\ \\
\displaystyle{(\disvh^{n+1})|_{\Gamma}=(\velfh^{n+1})|_{\Gamma}.}
\end{array}
\end{equation}
Similarly as in the previous step we see that the integral associated with the structure acceleration "lives" only in the strip $\Omega^h_S$, i.e. $\disvh^{n+1}=\disvth^{n+1}$ on the nodes inside the structure domain. Again, we can conclude that~\eqref{A2BisThick} is the fluid problem because the unknowns are the fluid velocity and the fluid pressure, while the structure inertia is taken into the account on the interface. This is crucial for the stability of the scheme.

To derive energy estimates, we take test functions $(\tfh,\tsh)=\Delta t(\velfth^{n+1},\disvth^{n+1})$ in \eqref{A1BisThick}, $(\tfh,\tsh)=\Delta t(\velfh^{n+1},\disvh^{n+1})$ in \eqref{A2BisThick} and sum the resulting expressions. We end up with the same energy estimates as in \cite{thick}, but with the following additional term (analogously as in Section \ref{sec:Stability}):
$$
I=\Delta t\int_{\Gamma}\str_h^n\norm(\disvh^{n+1}-\disvth^{n+1}).
$$
The problem is that now we do not have the thin structure inertia that would help us to deal with the problematic term. However, numerically we  still have some structure inertia  in the fluid step. Namely, after integration by parts of \eqref{A2BisThick} we obtain:
\begin{equation}\label{ThickInertia}
\displaystyle{\rho_s\int_{\Omega^h_S}\frac{\disvh^{n+1}-\disvth^{n+1}}{\Delta t}\cdot\tsh=\int_{\Gamma}(\str^n_h-\str^{n+1}_h)\norm\cdot\tfh.}
\end{equation}
Let us now take into account that $\tfh=\tsh|_{\Gamma}$ to derive the relation between the structure inertia and the fluid force on the interface $\Gamma$. First, we introduce some notation.

Let $\phi^h_i$, $i=1,\dots,m$ be the finite element functions on the interface $\Gamma$ and $\psi^h_i$, $i=1,\dots,m$, corresponding finite elements functions in the structure domain $\Omega_S$, i.e. $(\psi^h_i)_{|\Gamma}=\phi^h_i$ and $\psi_i$ are supported in $\Gamma\times (0,h)$. Notice that we consider only the structure elements that are associated with the nodes on the interface.
We denote by $A_h$ and $B_h$ the associated mass matrices, $A_h=(a^h_{ij})=(\int_{\Omega_S}\psi^h_i\psi^h_j)_{ij}$ and $B_h=(b^h_{ij})=(\int_{\Gamma}\phi^h_i\phi^h_j)_{ij}$.
Let $v=\sum_{i=1}^m v_i\psi^h_i$ be a function defined on the structure domain. Then its trace is given by $v|_{\Gamma}=\sum_{i=1}^m v_i\phi^h_i$. With a slight abuse of a notation we will identify function $v$ with vector ${\bf v}=(v_i)_{i=1,\dots,m}$. Furthermore, we have
$$
\|v\|^2_{L^2(\Omega_S)}=\sum_{i,j=1}^m v_iv_ja_{ij}=A_h{\bf v}\cdot {\bf v},
\qquad 
\|v_{|\Gamma}\|^2_{L^2(\Gamma)}=\sum_{i,j=1}^m v_iv_jb_{ij}=B_h{\bf v}\cdot {\bf v}.
$$
Moreover, notice that $\|A_h\|\approx h^3$ and $\|B_h\|\approx h^2$ because $\psi_i$ are $3d$ elements and $\phi_i$ are $2d$ elements, and their maximum triangle diameter is $h$.
Using the following equation
\begin{equation}\label{ThickInertia2}
\displaystyle{\rho_S\int_{\Omega^h_S}\frac{{\bf v}^{n+1}_h-\tilde{\bf v}^{n+1}_h}{\Delta t}\cdot\tfh=\int_{\Gamma}(\str^n_h-\str^{n+1}_h)\norm\cdot\tsh.},
\end{equation}
we obtain
$$
\rho_s(\frac{{\bf v}^{n+1}_h-\tilde{\bf v}^{n+1}_h}{\Delta t})_{|\Gamma}=A_h^{-1}B_h(\str^n_h-\str^{n+1}_h)\norm.
$$
Here we used the identification between functions and the coefficients vectors in order to define the operator on the right-hand side.
Therefore we have
$$
\Delta t\int_{\Gamma}\str^n_h\norm\cdot({\bf v}^{n+1}_h-\tilde{\bf v}^{n+1}_h)=\frac{\Delta t^2}{\rho_s}\int_{\Gamma}\str^n_h\norm\cdot A_h^{-1}B_h(\str^n_h-\str^{n+1}_h)\norm=\
\frac{1}{2}\frac{\Delta t^2}{\rho_s}\Big (\int_{\Gamma} (A_h^{-1}B_h \str^n_h)\norm\cdot\str^n_h\norm
$$
$$
-\int_{\Gamma} (A_h^{-1}B_h\str^{n+1}_h\norm)\cdot\str^{n+1}_h\norm
+\int_{\Gamma} (A_h^{-1}B_h( \str^n_h-\str^{n+1}_h)\norm)\cdot(\str^n_h-\str^{n+1}_h)\norm\Big ).
$$
Let us calculate the last term
$$
\frac{1}{2}\frac{\Delta t^2}{\rho_s}\int_{\Gamma} (A_h^{-1}B_h( \str^n_h-\str^{n+1}_h)\norm)\cdot(\str^n_h-\str^{n+1}_h)\norm=\frac{1}{2}\rho_s\int_{\Gamma}({\bf v}^{n+1}_h-\tilde{\bf v}^{n+1}_h)\cdot(B_h^{-1}A_h({\bf v}^{n+1}_h-\tilde{\bf v}^{n+1}_h))
$$
$$
=\frac{\rho_s}{2}\|{\bf v}^{n+1}_h-\tilde{\bf v}^{n+1}_h\|^2_{L^2(\Omega_S)}.
$$
Notice that the same term is obtained in the left-hand side by taking test functions $(\tfh,\tsh)=\Delta t(\velfh^{n+1},\disvh^{n+1})$ in \eqref{A2BisThick} and using identity~\eqref{identity} on the second term (the structure inertia).
Therefore this term is canceled with the same term from the left-hand side that comes from the structure inertia that is included in the fluid step. Moreover, $A^{-1}_hB_h$ is a positive-definite matrix and therefore one can proceed to obtain analogous stability and convergence estimates as in the thin structure case as long as the term $\frac{1}{2}\frac{\Delta t^2}{\rho_s}\|A^{-1}_hB_h\|$ stays bounded. Since $\|A_h^{-1}B_h\|\approx \frac{C}{h}$ we have the following stability condition:
$$
\Delta t^2\ltt h.
$$
More precisely we proved the following stability result:
\begin{theorem}\label{ThickStability}
Let $\{(\velfh^n, \disvth^n,  \disvh^n, \displh^n \}_{0 \le n \le N}$ be the solution of~\eqref{A1BisThick}-\eqref{A2BisThick}.
Then, the following estimate holds:
  \begin{gather}
  \mathcal{E}_f(\velfh^N)+\mathcal{E}_v(\disvh^N)+\mathcal{E}_s(\displh^N)  + \frac{\Delta t^2}{ \rho_s h} \n\str(\velfh^{N},p_h^{N}) \norm\n^2_{L^2(\Gamma)} +\frac{\rho_f \Delta t^2}{2}\sum_{n=0}^{N-1}\n  d_t \velfh^{n+1} \n^2_{L^2(\Omega_F)}
     \notag \\  
 +\frac{\Delta t^2}{2}\sum_{n=0}^{N-1} a_{ts}(d_t\displh^{n+1},d_t\displh^{n+1})
+ \mu \Delta t \sum_{n=0}^{N-1} \n \velfh^{n+1} \n^2_{F}  +\frac{\rho_s }{2}\sum_{n=0}^{N-1}\n \disvth^{n+1}-\disvh^{n} \n^2_{L^2(\Omega_S)}  
\notag \\
\ltt   \mathcal{E}_f(\velfh^0)+\mathcal{E}_v(\disvh^0)+\mathcal{E}_s(\displh^0)   + \frac{\Delta t^2}{\rho_s h} \n\str(\velfh^{0},p_h^{0}) \norm\n^2_{L^2(\Gamma)}
+\Delta t \sum_{n=0}^{N-1} \n p_{in/out} (t^{n+1})\n^2_{L^2(\Sigma)},\label{theorem:stabilityThick}
  \end{gather}
  where 
\begin{gather}
 \mathcal{E}_f(\velfh^n) = \frac{\rho_f}{2} \n\velfh^n\n^2_{L^2(\Omega_F)}, \quad \mathcal{E}_v (\disvh^n)= \displaystyle\frac{\rho_{s} }{2} \n \disvh^{n}\n^2_{L^2(\Omega_S)}, \quad  \mathcal{E}_s (\displh^n) =\frac{1}{2} a_{ts}(\displh^n, \displh^n).
\end{gather}
\end{theorem}
\begin{remark}
Using the obtained stability estimates in an analogous way as in Section~\ref{Sec:Convergence}, one can show that the scheme is convergent and its order of temporal accuracy is $\mathcal{O}(\frac{\Delta t}{\sqrt{h}})$. This is the same order of accuracy that is obtained by an alternative splitting strategy in \cite{fernandez2015convergence}. 
\end{remark}

\section{Numerical results}

In this section we focus on the verification of  the stability and convergence results of the  kinematically coupled $\beta$ scheme. We consider a benchmark problem similar to the one proposed in~\cite{Fernandez3}, belonging to a class of benchmark problems commonly used to validate FSI solvers. As in~\cite{Fernandez3}, we consider a two-dimensional test problem. The fluid domain is the rectangle $\domf=(0,L) \times (0,R)$ with $R=0.5$ cm, $L=5$ cm.  The top boundary corresponds to the fluid-structure interface, while symmetry conditions are prescribed at the bottom boundary. The fluid physical parameters are given by $\rho_f = 1$ g/cm$^3$ and $\mu$=0.035 g/cm s.  The flow is driven by the inlet time-dependent pressure data, which is a cosine pulse with maximum amplitude $p_{max} = 1.3333 \times 10^4$ dyne/cm$^2$ lasting for $t_{max}=0.003$ seconds, 
while the outlet normal stress is kept at zero:
 \begin{equation*}
 p_{in}(t) = \left\{\begin{array}{l@{\ } l} 
\frac{p_{max}}{2} \big[ 1-\cos\big( \frac{2 \pi t}{t_{max}}\big)\big] & \textrm{if} \; t \le t_{max}\\
0 & \textrm{if} \; t>t_{max}
 \end{array} \right.,   \quad p_{out}(t) = 0 \;\forall t \in (0, T).
\end{equation*}
 The problem is solved over the time interval [0,16] ms. 

\subsection{Fluid-thin structure interaction}
In this subsection we consider the interaction between a fluid and a thin structure. We model the structure elastodynamics using a generalized string model with the assumption of zero axial displacement:
\begin{gather}
\displ=(0, \eta_r)^T, \quad {\mathcal L}_S\displ^{n+1} = (0, C_0 \eta_r - C_1 \partial_{xx} \eta_r)^T \quad \textrm{with} \;\; C_0=\frac{E \epsilon}{R^2(1-\sigma^2)} \;\;\textrm{and} \;\; C_1=\frac{E \epsilon}{2(1+\sigma)},
\end{gather}
where $E$ is the Young's modulus and $\sigma$ is Poisson's ratio. The structure physical parameters are $\rho_s= 1.1$ g/cm$^3$, $\epsilon=0.1$ cm, $E=0.75 \cdot 10^6$ dyne/cm$^2$ and $\sigma=0.5$. To discretize the fluid problem in space, we use  the $\mathbb{P}_1$ bubble--$ \mathbb{P}_1$ elements for the velocity and pressure, and $\mathbb{P}_1$ elements to discretize the structure problem.

In order to verify the time convergence estimates from Theorem~\ref{MainThm}, we fix $h=L/640=0.0078$ cm and define the reference solution to be the one obtained with $\Delta t = 5 \cdot 10^{-6}$. Figure~\ref{F1}
shows the relative error between the reference solution and solutions obtained with $\Delta t = 5 \cdot 10^{-4}, 2.5 \cdot 10^{-4}, 1.25 \cdot 10^{-4}, 6.25 \cdot 10^{-5}$ and $3.125 \cdot 10^{-5}$ for the fluid velocity $\velfh^N$ in $L^2$-norm (left) and for the structure displacement $\displh^N$ in $\|\cdot\|_S$ norm (right) obtained at $T=10 $ ms. We compare the rate of convergence for the values of $\beta=0, 0.25, 0.5, 0.75$ and $\beta=1$. We observe that the case when $\beta=1$ leads to the optimal, first-order in time convergence, while sub-optimal convergence is obtained when  $\beta<1$. 
\begin{figure}[ht!]
 \centering{
 \includegraphics[scale=0.7]{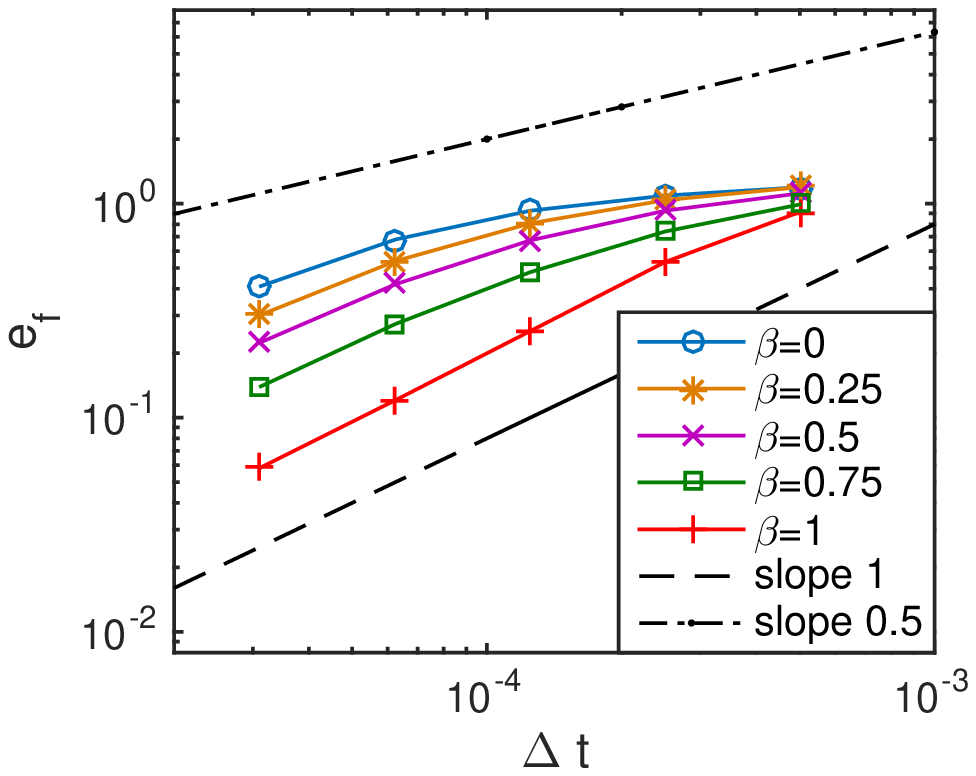}
\includegraphics[scale=0.7]{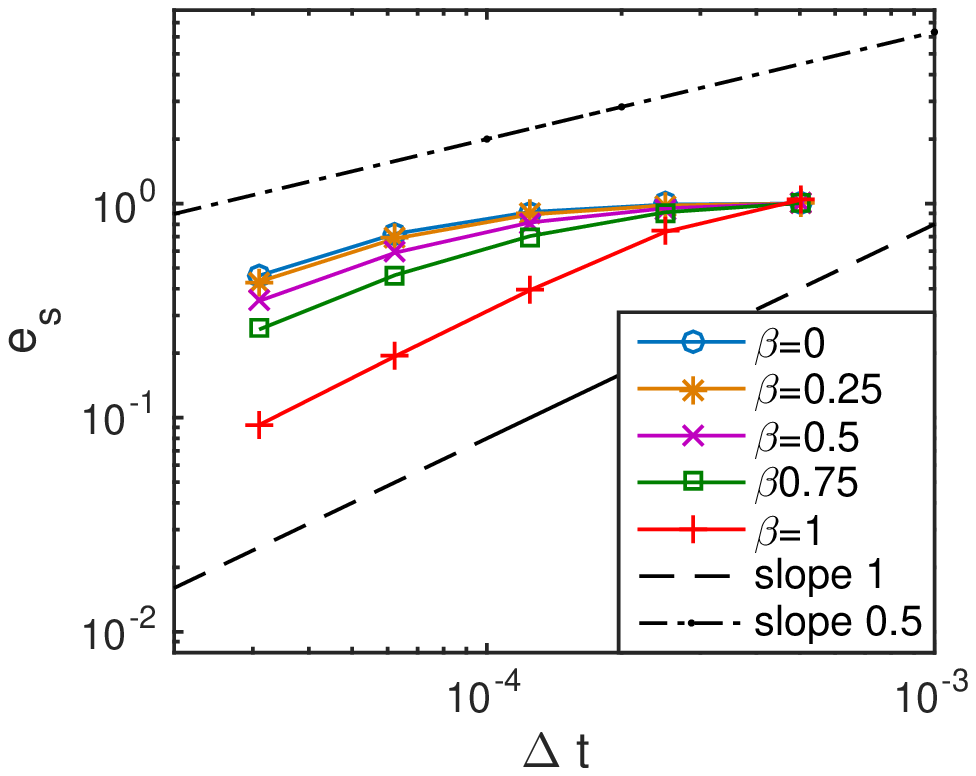} 
 }
 \caption{Time convergence  obtained at t=10 ms. Left: Relative error for fluid velocity in $L^2$-norm. Right: Relative error for  the structure displacement in $\|\cdot\|_S$ norm. Higher rate of convergence is observed in the case when $\beta=1$. }
\label{F1}
 \end{figure}
 
 \subsection{Fluid-thick structure interaction}
 In this subsection we model the interaction between a fluid and thick structure using algorithm~\eqref{A1BisThick}-\eqref{A2BisThick}. We assume that the thick structure elastodynamics is described by 
 \begin{equation}
a_{ts}(\displ,\ts)=2 \mu_s \int_{\Omega_S}{\bf D}(\displ): {\bf D} (\ts)+\lambda_s \int_{\Omega_S}(\nabla \cdot \displ) (\nabla \cdot \ts)+ C_{as} \int_{\Omega_S} \displ \cdot\ts,
\end{equation} 
where $\Omega_S = (0,L) \times (0,H),$ with $L=5$ cm, $H=0.1$ cm. The last term in the thick structure model is obtained from the axially symmetric model, and it represents a spring keeping the top and bottom boundaries connected~\cite{thick}.  The thick structure physical parameters are $\rho_s= 1.1$ g/cm$^3$, $\mu_s =2.586 \cdot 10^5$ dyne/cm$^2$, $\lambda_s =2.328 \cdot 10^6$ dyne/cm$^2$ and $C_{as}=4 \cdot 10^6$ dyne/cm$^4$. To discretize the fluid problem in space, we use  the $\mathbb{P}_1$--iso--$\mathbb{P}_2$ and $ \mathbb{P}_1$ elements for the velocity and pressure, and $\mathbb{P}_1$ elements to discretize the structure problem. 

We define the reference solution to be the one obtained with $h=0.00625$ (corresponding to the velocity mesh) and $\Delta t = 5 \cdot 10^{-6}$. To verify the convergence rate $\mathcal{O}(\frac{\Delta t}{\sqrt{h}})$ predicted in Remark 3, we consider two different scalings, $\Delta t = \mathcal{O}(h)$ and  $\Delta t = \mathcal{O}(h^{3/2})$. Figure~\ref{F3}
shows a comparison of the relative error between the reference solution and solutions obtained with $\Delta t = \mathcal{O}(h)$ and $\Delta t = \mathcal{O}(h^{3/2})$ for the fluid velocity $\velfh^N$ in $L^2$-norm (left) and for the structure displacement $\displh^N$ in $\|\cdot\|_S$ norm (right)  at $T=10 $ ms. In the case when $\Delta t = \mathcal{O}(h)$, we used time steps $\Delta t = 5 \cdot 10^{-4}, 2.5 \cdot 10^{-4}$ and $1.25 \cdot 10^{-4}.$ For $\Delta t = \mathcal{O}(h^{3/2})$, we used $\Delta t = 5 \cdot 10^{-4}, 1.76 \cdot 10^{-4}$ and $6.25 \cdot 10^{-5}.$ We observe that the suboptimal convergence is obtained when $\Delta t = \mathcal{O}(h)$, which is improved when $\Delta t = \mathcal{O}(h^{3/2})$. 
\begin{figure}[ht!]
 \centering{
 \includegraphics[scale=0.65]{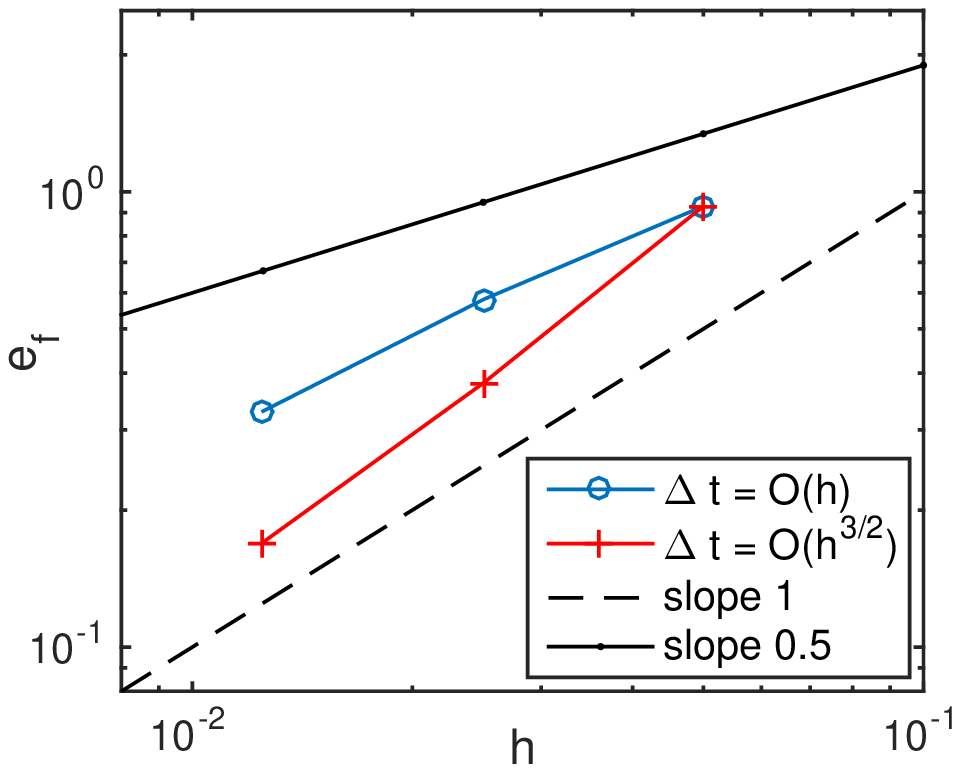}
\includegraphics[scale=0.65]{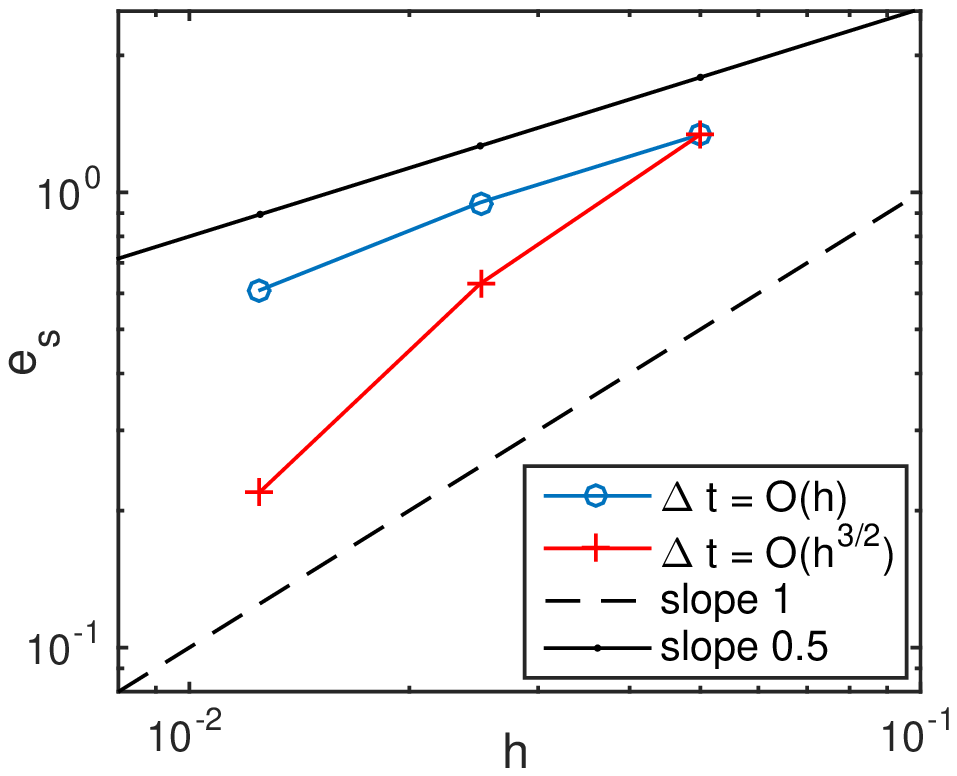} 
 }
 \caption{Relative error  obtained at t=10 ms using $\Delta t = \mathcal{O}(h)$ and $\Delta t = \mathcal{O}(h^{3/2})$. Left: Relative error for fluid velocity in $L^2$-norm. Right: Relative error for  the structure displacement in $\|\cdot\|_S$ norm. Higher rate of convergence is observed in the case when $\Delta t = \mathcal{O}(h^{3/2})$. }
\label{F3}
 \end{figure}

\section{Conclusions}

In order to complete the theory behind the kinematically coupled scheme and its variants, in this manuscript we analyze the stability and convergence properties of $\beta$-scheme. This is the first work that presents the \textit{a priori} error estimates which include the operator splitting error, and proves the optimal  $\mathcal{O}(\Delta t)$ convergence in time when $\beta=1$. Furthermore, we discuss the extension of our results to the fluid-thick structure interaction problem. 
Numerical experiments confirm the theoretical results.

\bibliography{kincouplconv}
\bibliographystyle{siam}

\end{document}